\documentclass[a4paper,12pt,english]{amsart}

\usepackage{amsmath,upref,amssymb}

\swapnumbers

\input xypic
\CompileMatrices

\RequirePackage{calrsfs}
\DeclareSymbolFont{rsfscript}{OMS}{rsfs}{m}{b}
\DeclareSymbolFontAlphabet{\mathrsfs}{rsfscript}
\renewcommand{\mathcal}{\mathrsfs}

\newcommand{\nc}{\newcommand}

\nc{\on}{\operatorname}

\nc{\DMO}{\DeclareMathOperator}
\numberwithin{equation}{subsection}

\theoremstyle{plain}
\newtheorem*{thm*}{Theorem}
\newtheorem*{lem*}{Lemma}
\newtheorem*{prop*}{Proposition}
\newtheorem*{cor*}{Corollary}

\theoremstyle{definition}
\newtheorem*{defn*}{Definition} 
\newtheorem*{conj*}{Conjecture}

\theoremstyle{remark}
\newtheorem*{rem*}{Remarks}

\makeatletter
\renewcommand{\p@enumi}{\thesubsection}
\renewcommand{\p@enumii}{\thesubsection\theenumi}
\makeatother

\newenvironment{num}{\renewcommand{\theenumi}{(\alph{enumi})}
                      
                                          \begin{enumerate} }
                    {\end{enumerate} }
 
 \newenvironment{subconds}{
                       
                        \begin{enumerate} }
                     {\end{enumerate} }

\nc{\be}{\begin{equation}}
\nc{\ee}{\end{equation}}

\nc{\bee}{\begin{equation*}}
\nc{\eee}{\end{equation*}}

\nc{\bs}{\begin{split}}
\nc{\es}{\end{split}}

\nc{\bc}{\begin{cases}}
\nc{\ec}{\end{cases}}

\nc{\bml}{\begin{multline}}
\nc{\eml}{\end{multline}}

\nc{\bmll}{\begin{multline*}}
\nc{\emll}{\end{multline*}}

\nc{\lb}{\label}

\newcommand{\mpair}[1]{\pair{\,#1\,}}

\newcommand{\mset}[1]{\set{\,#1\,}}

\newcommand{\pair}[1]{\langle #1\rangle}

\newcommand{\set}[1]{\{#1\}}

\nc{\mc}{{\mathcal}}
\nc{\mf}{{\mathfrak}}
\nc{\CA}{{\mathcal A}}
\nc{\CB}{{\mathcal B}}
\nc{\CC}{{\mathcal C}}
\nc{\CD}{{\mathcal D}}
\nc{\CE}{{\mathcal E}}
\nc{\CF}{{\mathcal F}}
\nc{\CG}{{\mathcal G}}
\nc{\CH}{{\mathcal H}}
\nc{\CI}{{\mathcal I }}
\nc{\CJ}{{\mathcal J }}
\nc{\CK}{{\mathcal K }}
\nc{\CL}{{\mathcal L}}
\nc{\CM}{{\mathcal M}}
\nc{\CN}{{\mathcal N}}
\nc{\CO}{{\mathcal O}}
\nc{\CP}{{\mathcal P}}
\nc{\CQ}{{\mathcal Q}}
\nc{\CR}{{\mathcal R}}
\nc{\CS}{{\mathcal S}}
\nc{\CT}{{\mathcal T}}
\nc{\CU}{{\mathcal U}}
\nc{\CV}{{\mathcal V}}
\nc{\CW}{{\mathcal W}}
\nc{\CX}{{\mathcal X}}
\nc{\CY}{{\mathcal Y}}
\nc{\CZ}{{\mathcal Z}}
\nc{\bbZ}{{\mathbb Z}}
\nc{\bbN}{{\mathbb N}}
\nc{\bbR}{{\mathbb R}}

\def\a{\alpha}
\def\b{\beta}
\def\g{\gamma}
\def\G{\Gamma}
\def\d{\delta}
\def\D{\Delta}
\def\e{\epsilon}

\def\L{\Lambda}

\def\Sig{\Sigma}

\def\t{\tau}

\def\to{\rightarrow}

\newcommand{\ssect}{\subsection}

\newcommand{\seq}{{\,\subseteq\,}}

\newcommand{\sreq}{{\,\supseteq\,}}

\newcommand{\sm}{{\,\setminus\,}}

\newcommand{\eset}{{\emptyset}}

\nc{\ol}{\overline}

\def\dotcup{\hskip1mm\dot{\cup}\hskip1mm}

\newcommand{\join}{\bigvee}

\newcommand{\meet}{\bigwedge}

\DMO {\Hom}{{\mathrm Hom}}

\begin{document}
\title{On the weak order of  Coxeter groups}

\author{Matthew Dyer}
\address{Department of Mathematics 
\\ 255 Hurley Building\\ University of Notre Dame \\
Notre Dame, Indiana 46556, U.S.A.}
\email{dyer.1@nd.edu}

\begin{abstract} This paper provides some  evidence for conjectural relations between extensions of (right) weak order on Coxeter groups,  closure operators on root systems, and Bruhat order. The  conjecture
focused upon here refines an earlier question as to whether the  set of initial sections of reflection orders, ordered by inclusion, forms a complete lattice.
 Meet and join in weak order are described  in terms of a suitable  closure operator.    Galois connections are defined from the  power set of $W$ to itself, under which maximal subgroups of certain  groupoids  correspond 
to certain complete meet subsemilattices of weak order. An analogue of weak order    for   standard parabolic subsets of any rank of the root system
is defined, reducing to the usual weak order in rank zero, and having some analogous properties in rank one (and conjecturally in general).  \end{abstract}

\maketitle
\section*{Introduction} Weak order is a   partial order on a Coxeter group $W$ which is of considerable importance in the basic  combinatorics of $W$.  For example, the maximal chains  in weak order from the identity element to  a given element are  in natural bijective correspondence with reduced expressions of that element. It is known  that  weak orders of Coxeter groups  are complete meet semilattices. 

  This paper gives descriptions of meet and join
(when existing) of elements in weak order in terms of a closure operator on the root system of $W$. The closure operator  used here  is the finest one (i.e. with the most closed sets)  for which, given any two roots in a closed set, any root expressible as a non-negative real linear combination of those two roots is also in that closed set, though some of  the results  hold for other closure operators as well. The paper also  introduce two Galois connections from the  power set of $W$ to itself, under which maximal subgroups of certain groupoids (generalizing those in \cite{HowBr}, and which we shall study in a series of other papers) correspond bijectively
to certain complete meet subsemilattices of weak order. It also  defines  an
analogue of weak order associated to   standard parabolic subsets of the root system (which are the unions of the positive roots with the root systems of  standard parabolic subgroups). The parabolic weak order  reduces to the standard weak order in the case of the trivial   parabolic subgroup, and it is shown to be   a complete meet semilattice
(with meet and join given by essentially the same formula in terms of the closure operator as for ordinary weak order) in the  case of a rank one parabolic subgroup.

These results have been obtained in attempting to refine  some questions and conjectures   on the structure of reflection orders of Coxeter groups and their initial sections raised in \cite[Remark 2.12]{DyHS1}  and \cite[Remark 2.14]{DyQuo}. Reflection orders and their initial sections have applications to study of  Bruhat order,  Hecke algebras and Kazhdan-Lusztig polynomials. The  original questions are recalled here, and    some of the refinements stated, since they  provide some of the principal motivations; additional related  results and conjectures  will be discussed in other papers.
The arrangement of this paper is as follows. Section \ref{S1} provides the statements of the main results. Section \ref{s2a} discusses the  conjectures.   Section \ref{S2} illustrates some of the main  results by the example of finite dihedral groups. Sections \ref{S3}--\ref{S9} are devoted to the proofs of the main results. Section \ref{S10} discusses two  other standard closure operators on root systems and the extent to which the  results are known to hold for them.

Throughout the paper, it is assumed  that the reader is familiar with basic properties of Coxeter groups and their root systems; standard  references for this background material are \cite{Bour} and \cite{Hum}.
For the basic properties of weak order, consult \cite{BjBr}. Several of the proofs  proceed by  reduction to the case of dihedral groups;  the easy  verifications of the necessary facts in the dihedral case are usually omitted.

\section{Statement of results} \lb{S1}

 \subsection{Posets and lattices}\lb{ss1.1} Fix the following standard terminology for partially  ordered sets, which are called  posets in this paper.  See for example \cite{DavPr}.  A   lattice is a non-empty poset 
  $(L,\leq)$  in which any two elements $x,y\in L$ have a least  upper bound (join) $x\vee y$ and a
 greatest lower bound (meet) $x\wedge y$. 
   A lattice $L$ is said to be complete  if  every   subset $X$ of $L$ has a  join $\join X$ and a meet $\meet X$ (this implies $L$ has a minimum element and a maximum element).

 An ortholattice is a lattice which has  a maximum element $\top$, minimum element $\bot$ and is  
 equipped with an order-reversing involution $x\mapsto x^{\perp}$ such that $x\vee x^{\perp}=\top$ and $x\wedge x^{\perp}=\bot$ for all $x$ in the lattice.  
A complete ortholattice is an ortholattice
 which is complete as lattice. 

  A complete meet semilattice is a
 poset   in which every non-empty subset $X$ has a greatest lower bound $\meet X$; in particular, a non-empty complete meet semilattice has a minimum element.
 By a complete meet subsemilattice of a complete meet semilattice $L$, we mean a subset $X$ of $L$ such for any non-empty subset $Y$ of $X$, the meet $\meet Y$ of $Y$ in $L$ is in $X$ (and is therefore the meet of $Y$ in $X$). If $X$ is non-empty, its minimum element need not coincide with that of $L$ (according to our conventions).

\ssect{Coxeter groups and root systems} \lb{ss1.2} Let $(W,S)$ be a  Coxeter system with standard length function $l=l_{W}$ and set of reflections $T=\mset{wsw^{-1}\mid w\in W,s\in S}$.  Assume without loss of generality  that $W$ is the real  reflection group associated as in \cite{Sd}   to a  based
root system $\Phi$  in a real vector  space $V$ (this allows the standard root system of \cite{Bour} and  \cite{Hum}).
Let $\Pi$ be the standard  set of simple roots corresponding to the simple reflections $S$ of $W$, and  let  $\Phi_+$  be the set of positive roots corresponding to $\Pi$.  Abbreviate $\Phi_{-}:=-\Phi_{+}$. Denote the reflection in a root $\alpha\in \Phi$ by  $s_\alpha\in T$. For  $z\in W$, let \[\Phi_{z}=\Phi_{z,W}:=\Phi_{+}\cap z(\Phi_{-})=\mset{\alpha\in \Phi_{+}\mid l(s_{\a}z)<l(z)}.\]  It is well known that  if $z=s_{\a_{1}}\ldots s_{\a_{n}}$ is any reduced expression, then  
\be\lb{eq1.2.1}
\Phi_{z}=\set{\a_{1},s_{\a_{1}}(\a_{2}),\ldots,s_{\a_{1}}\cdots s_{\a_{n-1}}(\a_{n})}, \qquad \vert \Phi_{z}\vert =l(z)=n.\ee Write $\Phi'_{z}:=\Phi_{+}\sm\Phi_{z}$.  Note that for $z\in W$,   $\Phi'_{z}$ is finite if and only if  $W$ is finite, in which case $\Phi'_{z}=\Phi_{zw_{S}}$ where $w_{S}$ is the longest element of $W$.

\ssect{Weak order} \lb{ss1.3} The weak (right) order $\leq$ on $W$ is the partial order on $W$ 
defined by $x\leq y$ if and only if  $l(y)=l(x)+l(x^{-1}y)$. It is also known that $x\leq y$ if and only if  $\Phi_{x}\seq \Phi_{y}$.
In particular, for any $x,y\in W$,  the equality  $\Phi_{x}=\Phi_{y}$ holds  if and only if  $x=y$. 

For any fixed $y\in W$, the set $\mset{x\in W\mid x\leq y}$ is finite. It is known that   $W$ in weak order is a  complete meet semilattice (see \cite[3.2]{BjBr}).

\ssect{The $2$-closure operator on a root system} \lb{ss1.4} Denote the power set of a set $X$ as $\CP(X)$. A closure operator on the set $X$ is a  function $c\colon \CP(X)\rightarrow \CP(X)$ such that 
 $c(c(A))=c(A)$ and $A\seq c(A)$ for $A\seq X$,  and  $A\seq B\seq X$ implies $c(A)\seq c(B)$. One  calls the  subsets of $X$ of the form $c(A)$ for $A\seq X$   the closed sets (of $c$). The closure $c(A)$ of $A$  is then the intersection of all closed subsets of $X$ containing $A$.

 Following \cite[Remark 2.12]{DyHS1}, introduce a closure operator on $\Phi$ as follows.
 Say that a subset $\Gamma$ of $\Phi$ is $2 $-closed  if  for any $\a,\b\in \G$
we have $\mset{a\a+b\b\mid a,b\in \bbR_{\geq 0}}\cap \Phi\subseteq \G$.
The $2 $-closed sets are the closed sets of the closure operator on $\Phi$ for which the   closure $\ol \G$   of a subset  $\G$ of $\Phi$  is  the intersection of all 
 $2 $-closed subsets of $\Phi$ which contain $\G$. Henceforward,  the $2 $-closed subsets of $\Phi$ are usually called  closed sets, unless other closure operators are under  simultaneous consideration.
 
 It is crucial for the purposes of this paper  that $\Phi_{+}$ is closed and for $x\in W$, $\Phi_{x}$ and $\Phi_{x}'$ are closed (see Lemma \ref{ss3.1}).

\begin{rem*}  The $2$-closure operator was called $\bbR $-closure in \cite{Pil}; the name $2$-closure   emphasizes its rank two nature.   See also \cite{DyRig} and Section \ref{S10} for
more  about this closure operator.  \end{rem*}

\ssect{Weak order and $2$-closure}\lb{ss1.5} The first main result   affords  a new proof that $W$ in  weak order is a  semilattice and a description of its meet and join in terms of the closure operator.

\begin{thm*} In weak order $\leq$, $W$ is a complete meet semilattice. The  join (when existing) and meet   of a non-empty subset  $X$ of $W$are given as follows:
  \begin{num} \item  The join $y:=\join X$ exists in $(W,\leq)$ if and only if  $X$ has an upper bound in $W$, in which case $\Phi_{y}=\ol{\cup_{x\in X}\Phi_{x}}$.
  \item If $y:=\meet X$  then  $ \Phi'_{y}=\ol{\cup_{x\in X} \Phi'_{x}}$.
  \end{num}\end{thm*}
  \begin{rem*} Subsequent conjectures of this paper would  imply that the join in (a) exists if and only if  
$\ol{\cup_{x\in X}\Phi_{x}}$ is a finite set, but this remains an open question.\end{rem*}
\ssect{Order isomorphisms} \lb{ss1.6} The order isomorphism in part  (a) of the following Corollary is well-known  (see \cite[Proposition 3.1.6]{BjBr}). 
Part (b), which is a simple consequence of Theorem \ref{ss1.5}, shows that the domain of this order isomorphism  is closed under formation of  those joins which exist in $W$; it  is  highlighted since it  will play a fundamental role in subsequent papers. 
\begin{cor*} Let $x\in W$. Then 
\begin{num}\item  The map $u\mapsto xu\colon W\rightarrow W$ restricts to an order isomorphism between
the subposet  (actually an order ideal)  $\mset{u\in W\mid l(xu)=l(x)+l(u)}$  and the subposet  (an order  coideal)
$\mset{z\in W\mid x\leq z}$ of $(W,\leq)$.
\item If $U\seq W$ with $l(xu)=l(x)+l(u)$ (i.e. $\Phi_{u}\cap \Phi_{x^{-1}}=\eset$) for all $u\in U$, and $y:=\join U$ exists in $W$,  
then  $l(xy)=l(x)+l(y)$ (i.e. $\Phi_{y}\cap \Phi_{x^{-1}}=\eset$).\end{num}\end{cor*}
\ssect{Bruhat order and $2$-closure} \lb{ss1.7}  Theorem \ref{ss1.5}(a)  and the Corollary are proved   in Section \ref{S4}, after giving  some preliminaries in  Section \ref{S3}.   A more general statement (Theorem \ref{ss6.1}) than Theorem \ref{ss1.5}(b) is proved   in Section \ref{S6}, after establishing 
the  next Lemma in Section \ref{S5}. To state the Lemma, the following  notation will be used.
For $n\in \bbZ$ and $x\in W$,  define \[\Phi_{x,n}:=\mset{\alpha\in \Phi_{+}\mid l(s_{\alpha}x)=l(x)+n}.\] 
  Note that $\Phi_{x,n}=\eset $ unless $n$ is odd, and that $\Phi_{x}=\dotcup_{n<0}\Phi_{x,n}$ and $\Phi'_{x}=\dotcup_{n>0}\Phi_{x,n}$ (here and later,  the symbol ``$\dotcup$''  indicates that a union is one of disjoint sets).

  The following  result is reminiscent of the Krein-Milman theorem, but note that $2$-closure is not a ``convex'' (i.e. anti-exchange) closure operator in general (see
Section \ref{S10}).
   \begin{lem*} Let $\G\seq\Phi_{+}$ and $x\in W$.  Then
\begin{num}\item  $\ol{\G}=\Phi_{x}$ if and only if  $\Phi_{x,-1}\seq \G\seq\Phi_{x}$.\item 
 $\ol{\G}=\Phi'_{x}$ if and only if  $\Phi_{x,+1}\seq \G\seq\Phi'_{x}$.  
 \end{num}
\end{lem*}

\ssect{Closure after adjoining a root }\lb{ss1.8} 

In Section \ref{S7}, the  following is proved  using the above Lemma and Theorem \ref{ss1.5}.
\begin{thm*} Let $x\in W$ and $\a\in \Phi_{+}$.
\begin{num} \item Suppose that $l(s_{\a}x)=l(x)+1$ and also that   there is some 
$v\in W$ with $\Phi_{x}\cup\set{\a}\seq \Phi_{v}$.
Then there is a minimum element $y=x\vee s_{\a}x$ of the set  $\mset{z\in W\mid x\leq z, {\a}\in \Phi_{z}}$.
One has  $\ol{\Phi_{x}\cup\set{\a}}=\Phi_{y}$
and  $\Phi_{s_{\a}y}=\Phi_{y}\sm\set{\a}=\ol{(\Phi_{x}\cup\Phi_{s_{\a}x})\sm\set{\a}}$. 
In particular, $y>s_{\a}y=ys_{\t}$ for some $\t\in \Pi$. 
\item Suppose that $l(s_{\a}x)=l(x)-1$. Then there is a maximum element $y=x\wedge s_{\a}x$ in
$\mset{z\in W\mid z\leq x, \a\not\in \Phi_{z}}$. One has 
$\Phi'_{y}=\ol{\Phi'_{x}\cup\set{\a}}$ and $\Phi'_{s_{\a}y}=\Phi'_{y}\sm\set{\a}=
\ol{(\Phi'_{x}\cup\Phi'_{s_{\a}x})\sm\set{\a}}$. In particular, $y<s_{\a}y=ys_{\t}$ for some $\t\in \Pi$.\end{num}
  \end{thm*}
\begin{rem*}  The minimum and  maximum elements    are  taken with respect to weak order.
Note that the hypotheses of (a) hold for any $x\in W$ and $\a\in\Pi \sm \Phi_{x}$ for which $x\vee s_{\a}$ exists; then $y=x\vee s_{\a}$. The hypotheses of (b) hold for any
$x\in W$ and $\a\in \Pi\cap \Phi_{x}$.    \end{rem*}

\subsection{Galois connection between subgroups and  subsemilattices} \lb{ss1.9} Define a relation $R$ on $W$ by $xRz$ if and only if  $z(\Phi_{x})=\Phi_{x}$  for $x,z\in W$.  As would any relation on $W$,  $R$ defines a Galois connection (see for instance \cite{MacL} and \cite{DavPr} for general background on Galois connections)  from $\CP(W)$ to itself as follows. Consider the  two maps $X\mapsto X^{\dag}$ and
$Z\mapsto Z^{*}$ from $\CP(W)\rightarrow\CP(W)$ defined by 
\bee X^{\dag}:=\mset{z\in W\mid xRz \text{ for all $x\in X$}}, \qquad Z^{*}:=\mset{x\in W\mid xRz\text{ for all $z\in Z$}}.\eee
Ordering $\CP(W)$ by inclusion, the maps are order-reversing and satisfy
$Z\seq X^{\dag}$ if and only if  $X\seq Z^{*}$ i.e. they give a Galois connection. 
As with any Galois connection, there are associated families  of stable sets (often called closed sets) for the composite maps:
\[\CW_{*}:=\mset{\CL\in \CP(W)\mid \CL^{\dag *}=\CL}=\mset{Z^{*}\mid Z\in \CP(W)}\]
and \[\CW_{\dag}:=\mset{\CG\in \CP(W)\mid \CG^{*\dag }=\CG}=\mset{X^{\dag}\mid X\in \CP(W)}.\] Well known properties of Galois connections imply that the  restriction of $*$ to a map $\CW_{\dag}\rightarrow \CW_{*}$ is a bijection
with inverse given by the restriction of $\dag$ to a map $\CW_{*}\rightarrow \CW_{\dag}$.
Also,   $\CW_{\dag}$ and $\CW_{*}$ are  complete lattices, dual under the above bijection, 
with meet given by intersection of subsets  of $W$.  In Section \ref{S8},  the following is proved 
and   analogous facts are   given  for the Galois connection determined similarly by the relation $R'$ on $W$ with
$xR'z$ if and only if  $z(\Phi'_{x})=\Phi'_{x}$. \begin{thm*} \begin{num}\item One has \bee xRz\iff (x\vee z=zx \text{ and $\Phi_{x}\cap \Phi_{z}=\eset$)}\iff \Phi_{zx}=\Phi_{z}\dotcup \Phi_{x}\eee
\item If $xRz$ then $xRz^{-1}$ and  $l(zx)=l(z)+l(x)$.
\item The elements of $\CW_{\dag}$ are subgroups of $W$.
\item The elements of $\CW_{*}$ are complete meet subsemilattices of $(W,\leq)$ with $1_{W}$ as minimum element.
\item If $L\in \CW_{*}$, then  for any  subset $X$ of $L$
which has an upper bound in $W$, its join $x:=\join X$ in $W$ is an element of $L$ (and so $x$ is the least upper bound of $X$ in $L$).
 \end{num}\end{thm*}
 \subsection{Groupoids associated to the Galois connections} \lb{ss1.10} Adopt  here the point of view  that a groupoid is a category, with a set of objects, in which every morphism is an isomorphism. 
   Each element of $\CW_{\dag}$   is obviously a maximal subgroup of 
 (i.e. the automorphism group of  an object of) a groupoid
 with objects $I$-indexed  families $X=(x_{i})_{i\in I}$ of elements  of $W$, for a suitable index set $I$,   morphisms \[\Hom(X,Y)\cong\mset{z\in W\mid z(\Phi_{x_{i}})=\Phi_{y_{i}} \text{ for all $i\in I$}}\]  and composition induced naturally by multiplication in $W$. Similarly for   $R'$. \begin{rem*}
 In the case of the relation $R$, the full subgroupoids  with objects the indexed families  of elements of  $S$ (more precisely, their  variants using subsets of $S$ instead of indexed subfamilies)  were studied in \cite{HowBr}. The  groupoids  defined above will be further generalized and studied in a series of   subsequent   papers.
  \end{rem*}
 
 \subsection{Coclosed and biclosed subsets of roots} \lb{ss:term}
   For any closed  subset $\L$ of $\Phi$, say that a subset  $\G$ of $\L$
  is coclosed  in $\L$ if $\L\sm\G$ is closed. Say that $\G\seq \L$ is biclosed in $\L$ if 
  $\G$ is  closed in $\L$ and $\L\sm \G$ is closed in $\L$. Let $\CB(\L)$ denote the set of biclosed subsets of $\L$. Order $\CB(\L)$  by inclusion. Note that $\CB(\L)$ is a complete poset, in the sense that it has a minimum element $\eset$ and the union of any directed family of elements of $\CB(\L)$ is in $\CB(\L)$. For any closure operator $a$ on $\Phi$, define notions of $a$-coclosed sets, $a$-biclosed sets etc in a similar manner as for $2$-closure.
  
  For example, the finite biclosed subsets of $\Phi_{+}$ are the sets $\Phi_{x}$ for $x\in W$,
by Lemma \ref{ss3.1}. Since this paper deals mostly with subsets of $\Phi_{+}$,  the terminology  will be  slightly abbreviated in that case: by a  ``coclosed set'' (resp., ``biclosed set'') is meant a set which  is a  coclosed subset of $\Phi_{+}$  (resp., a biclosed subset of $\Phi_{+}$).

\subsection{Parabolic weak orders}Standard parabolic subsets of $\Phi$ are now defined, following the usual definition for Weyl groups \cite[Ch VI, \S 1, Prop 20]{Bour}. In the general  context of this paper, the analogues of the conditions of loc. cit are no longer all equivalent
(e.g. for an infinite dihedral group); the definition here of standard parabolic subset is modeled  on condition (iii) of loc. cit. and the more general definition of quasiparabolic subset in \ref{ss1.14} is based on the condition (i) there.

 For $J\seq S$, there is  a standard parabolic subgroup $W_{J}:=\mpair{J}$ generated by $J$,  and its root system
$\Phi_{J}=\mset{\a\in \Phi\mid s_{\a}\in W_{J}}$. Set $\L:=\L_{J}:=\Phi_{+}\cup\Phi_{J}$. Call $\L_{J}$ the standard parabolic subset of $\Phi$ (of rank $\vert J\vert$) associated to $J$ (or associated  to  $\Phi_{J}$). Let $\CL=\CL_{J}$ denote the set of all finite biclosed  subsets   $\G$ of $\L_{J}$. Order  $\CL_{J}$ by inclusion, and call the resulting poset  the parabolic weak order  associated to $J$. Note that $\CL_{\eset}=\mset{\Phi_{x}\mid x\in W}$ naturally identifies with $W$ in weak  order via $x\leftrightarrow\Phi_{x}$. Define $\tau\colon \CP(\L_{J})\rightarrow \CP(\Phi_{J}) $ by $ \tau(\G):=\G\cap \Phi_{J}$.

\subsection{$2$-closure and  rank one parabolic weak order} Using the previous
\lb{ss1.11}  results, the following fact is  proved in Section \ref{S9}.\begin{thm*} Suppose above that $\vert J\vert =1$
Then $\CL_{J}$ is a complete meet semilattice of subsets of $\L:=\L_{J}$. Further,  \begin{num}
\item If $Y\seq \CL_{J}$ has an upper bound in  $\CL_{J}$, then it has a join $\D:=\join Y$ in $\CL_{J}$ given by $\D=\ol{\cup_{\G\in Y}\G}$.
\item
The meet  $\D$  of a subset $Y$ of $\CL_{J}$ is given by $\L\sm \D= \ol{\cup_{\G\in Y}(\L\sm \G)}$.
\item  The map  $\G\mapsto \tau(\G)$ 
 maps $\CL$ into the lattice   of finite biclosed subsets of  $\Phi_{J}$ (with biclosed sets ordered by inclusion), preserving meets and those joins which exist. 
 Moreover, $\t (\ol \G)=\ol{\tau(\G)}$ for any $\G\seq \L$.
\end{num}\end{thm*}
\begin{rem*} The analogue of the above Theorem with $J=\eset$  is essentially equivalent to 
Theorem  \ref{ss1.5}. The analogous statement could be conjectured to hold for any $J$, but then one may have $\CL_{J}=\set{\eset}$
(e.g for $J=S$ and $W$ infinite dihedral) and the conjecture in that form is  chiefly of interest for finite $W$ (for which it is open if $\vert J\vert >1$). A more general conjecture without this difficulty is formulated in Section \ref{s2a}.\end{rem*} 
\section{Conjectures}\lb{s2a}
The results of this paper  have been obtained in investigating an extensive set   of  questions and conjectures 
involving generalizations of basic combinatorics of Coxeter groups. In this section, some  of the conjectures most closely related to the contents of this paper, and not requiring much additional background, are stated.

The conjectures originated in studying applications of reflection orders, the original motivation for the definition of which  was to extend symmetry amongst 
  structure constants of Iwahori-Hecke algebras from the case of finite Coxeter groups
  to general Coxeter groups; this required a substitute  for the reduced expressions of the longest element, which was provided by the reflection orders.

\ssect{Reflection orders and initial sections} 
 A reflection order of $\Phi_{+}$ is defined as 
 a total order $\preceq$ of $\Phi_{+}$ such that for $\a,\b,\g\in \Phi_{+}$ with 
 $\a\prec\g$ and $\b\in \bbR_{>0}\a+\bbR_{> 0}\g$, we have $\a\prec \b\prec \g$. See \cite{BjBr} for a discussion of them and some applications.
 Under transport of structure from $\Phi_{+}$ to $T$ using the natural bijection $\a\mapsto s_{\a}$, reflection orders of $\Phi_{+}$ correspond   to the  reflection  orders of $T$ in the sense of \cite{DyHS1}
 (which are combinatorial, in that they may be defined purely in terms of $(W,S)$).
  
 Abbreviate the set of biclosed subsets of $\Phi_{+}$ as  $\CB:=\CB(\Phi_{+})$.  Define an admissible order of $\G\in \CB$ to be  a total order $\preceq$ of $\G$ all of the initial sections of which are  in $\CB$ (where  for any totally ordered set $P$,  an initial section  of $P$ is by definition  an order ideal i.e. a  subset $I$ of $P$ such that $x\leq y$ for all $x\in I$ and $y\in P\sm I$.).
 
 By straightforward reduction to the case of dihedral groups, it follows  that a  total order $\preceq$ of $\Phi_{+}$ is an admissible order of $\Phi_{+}$ if and only if  it is a reflection order of $\Phi_{+}$.
 Let $\CA$ denote the set of all subsets $\G$  of $\Phi_{+}$  for which there exists some reflection order $\preceq$ of $\Phi_{+}$ with $\G$ as an initial section. It is easily checked  from the definitions  that $\CA\seq \CB$. Attached to each element of $\CA$, there is  a ``twisted Bruhat order'' of $\CA$ as in \cite{DyHS2}; the definition of these orders may be extended  to the elements of $\CB$ (\cite{Edg}) but the more general orders are 
 not known to have such strong properties as those from elements of $\CA$.  On the other hand, $\CB$ has many useful properties not obviously shared by  $\CA$; for example, $\CB$ is closed under arbitrary directed unions, but this is not known for $\CA$.
 
\subsection{Reflection orders and maximal chains of biclosed sets} \lb{ssinsect} As will be  explained in \ref{basconj}, the following conjecture is naturally  suggested by naive analogy with the most basic facts about combinatorics of Coxeter groups.  
 
\begin{conj*} \begin{num}\item  $\CA=\CB$ \item For any $\G\in \CA$ and  any totally ordered (by inclusion) subset $\CI$ of  $\CA\cap \CP(\G)$, there is an admissible order $\preceq$ of  $\G$ such that every element of $\CI$ is an initial section of $\preceq$.
\end{num}\end{conj*}
\begin{rem*} Part (a) of the  conjecture is equivalent to the conjecture \cite[Remark 2.12]{DyHS1}, and the special case $\G=\Phi_{+}$ of   (b) is   equivalent  to (a  positive answer for) a question raised  in  \cite[Remark 2.14]{DyQuo}.
Given (a), (b) is equivalent to its own special case with $\G=\Phi_{+}$. The conjecture is known for finite Coxeter groups (see below) and will 
be proved in the special case of affine Weyl groups in another paper.\end{rem*}

\ssect{Admissible orders as generalized reduced expressions} \label{basconj} It follows from \cite[Lemma 2.11]{DyHS1},  \cite[Example 2.2]{DyQuo} and Lemma \ref{ss3.1} that
 \bee \mset{B\in \CA\mid \vert B\vert \text{ \rm is finite}}=\mset{\Phi_{w}\mid w\in W}=\mset{B\in \CB\mid \vert B\vert \text{ \rm is finite}}\eee and that the admissible orders of $\Phi_{w}$ are in natural
 bijective correspondence with the reduced expressions of $w$ as follows:
 to a reduced expression $w=s_{\a_{1}}\cdots s_{\a_{n}}$ of $w$, one attaches an admissible order $\preceq$ of $\Phi_{w}=\set{\beta_{1},\ldots, \beta_{n}}$ given by $\b_{1}\prec \ldots \prec\b_{n}$ where $\b_{i}:=s_{\a_{1}}\ldots s_{\a_{i-1}}(\a_{i})$ (compare
 \eqref{eq1.2.1}). 
Thus, $W$ may be identified  with the subset $\set{\Phi_{w}\mid w\in W}$ of $\CB$,    the elements of $\CB$ may be viewed
 as   generalized   elements of $W$, and the notion of   admissible order of an element of $\CB$ may be regarded as  a generalization of the notion  of reduced  expression of an element of $W$.  In the case of finite $W$, the generalized notions are precisely equivalent to the original ones,  but this is not true for any infinite Coxeter group. In fact, it  it may be shown that there are examples of infinite, finitely generated Coxeter groups $W$ for which
 $\CA$ (and  hence  $\CB$) is uncountable.  In any case, the partial order of $\CB$ by inclusion naturally generalizes the weak order on $W$,
 and will be  called  the extended weak order of $W$.
 
Given Conjecture  \ref{ssinsect}(a),  Conjecture \ref{ssinsect}(b)  is  equivalent (by Zorn's lemma) to the statement  that for $\G\in \CB$,  the  map taking an  admissible order of $\G$ to the set of its initial sections gives a bijection between the admissible orders of $\G$ and the   maximal totally ordered subsets of $\mset{\D\in \CB\mid \D\seq \G}$. 
 The  conjectures  together  therefore  generalize the statements that every element $w$ of $W$ has a reduced expression and that the reduced expressions of an element $w$ of $W$ are in natural bijective correspondence with maximal chains  from $1$ to $w$ in weak order of $W$.   
 \subsection{Reflection subgroups and closed sets of roots}   One might ask how much  of the standard  combinatorics involving  elements of  finite Coxeter groups and their reduced expressions, when suitably reformulated, extends  to elements of $\CB$ and their admissible orders. For example, since weak order on $W$ is a lattice if $W$ is finite, it is natural to ask
 if $\CB$, ordered by inclusion, is also a lattice in general; this question was raised in  \cite[Remark 2.14]{DyQuo}. 
A more precise and more general  version of this question is formulated as a conjecture below, extending Theorems \ref{ss1.5} and \ref{ss1.11}. In order to do this,  we first record a simple Lemma.

  Note that the reflection subgroups $W'$ of $W$ with closed root subsystem $\Phi_{W'}:=\mset{\a\in \Phi
\mid s_{\a}\in W'}$ constitute a complete meet subsemilattice of the complete lattice of reflection subgroups.  An arbitrary reflection subgroup $W'$ of $W$ need not have  closed root system
(e.g.  in type $B_{2}$).    However, one does have the following:
 \begin{lem*}
Let $\Xi$, $\L$ be  subsets of $\Phi$ such that $\Xi$, $\L$ are both closed, $\Xi=-\Xi$ and $\Xi\seq \L$. Then $\Xi$ is the root system 
of the reflection subgroup $W':=\mpair{s_{\a}\mid \a\in \Xi}$ and the natural $W$-action on $\Phi$ restricts to a $W'$-action on $\L$.\end{lem*}
\begin{proof}  For $\a\in \Xi$ and $\b\in \L$, 
   $s_{\a}(\b)\in (\b+\bbR \a)\cap \Phi\seq \L$ since $\L$ is closed, $\b\in \L$ and $\set{\a,-\a}\seq \Xi\seq \L$. This proves that $\L$ is $W'$-stable. The statement that $\Xi$ is the root system of $W'$ follows by taking $\L=\Xi$.\end{proof}

\subsection{Conjecture on quasiparabolic weak order and $2$-closure} \lb{ss1.14} 
 A subset $\L$ of $\Phi$ will be called  quasiparabolic 
  if $\L$ is closed and $\L\cup -\L=\Phi$.  The standard parabolic subsets are obviously quasiparabolic. There is a classification of quasiparabolic subsets (and more generally, of elements of $\CB(\Phi)$) in terms of elements of $\CB(\Phi_{+})$ and additional combinatorial data, which  won't be discussed here. In  \cite{BF},  analogues of quasiparabolic sets in (possibly infinite) oriented matroids are called large convex sets.
  
   Fix a quasiparabolic subset $\L$ of $\Phi$. 
   By the preceding Lemma, $\Psi:=\Psi(\L)=\L\cap -\L$ is 
  the root system of a reflection subgroup $W'=W(\L):=\mpair{s_{\a}\mid \a\in \Psi}$ of $W$. Note that  $W'$ acts on
  $\L$  by $(w,\g)\mapsto w(\g)$, preserving $\Psi$. This $W'$-action  on 
  $\L$ (resp., $\Psi$) obviously  induces a $W'$-action by order automorphisms on $\CB(\L)$ (resp., $\CB(\Psi)$).
  
 Define the map $\tau\colon \CP(\L)\rightarrow \CP(\Psi)$ by $\t(\G):=\G\cap \Psi$
 for $\G\seq \L$.
 Call $\t(\G)$ the type of  $\G$. One clearly has $\t(w(\G))=w(\t(\G))$ for $w\in W$,
 and $\t(\G)\in \CB(\Psi)$ if $\G\in\CB(\L)$.

   \begin{conj*} Let $\L$ be a quasiparabolic subset  of $\Phi$  (for example, a standard parabolic subset such as $\Phi_{+}\cup \Phi_{J}$ for $J\seq S$). Set $\Psi:=\Psi(\L)$
   and $W':=W(\L)$.
    Then 
\begin{num}
 \item  The set $\CB(\L)$ of  biclosed subsets of  $\L$ 
is a complete ortholattice. The join  of a family $X$  of  biclosed subsets of $\L$ is given by  $\join X=\ol{\cup_{\G\in X}\G}$, and the ortholattice complement is just set complement in $\L$.
\item The restriction of $\tau$ to a map $\CB(\L)\rightarrow \CB(\Psi)$ is a $W'$-equivariant  morphism of complete ortholattices  (i.e. it preserves $W'$-action, preserves  arbitrary meets and joins, and preserves complements)
\item  If $\G$ is  coclosed in $\L$, then  $\ol{\G}$ is  biclosed in $\L$.
\item If $\G\seq \L$, then $\tau(\ol{\G})=\ol{\tau(\G)}$.\end{num}
\end{conj*}
\begin{rem*} There are several dependencies amongst parts of this conjecture, and reductions of its parts to superficially weaker statements. 
For example,  using that $\CB(\L)$ is a complete poset, the conjecture (a) above follows easily from the special case of
(c) according to which $\ol{\G\cup\D}\in\CB(\L)$ if
$\G,\D\in \CB(\L)$.  Most of this paper is  devoted to checking parts of the conjecture involving subsets of $\L$ which are either finite or cofinite in $\L$, where $\L$ is a (rank one or zero) standard parabolic subset. One can obtain 
by similar arguments or reduction to results here some more  weak supporting  evidence involving
subsets of $\L$ which are neither finite nor cofinite, and for more general quasiparabolic subsets $\L$. The conjecture is open for all infinite irreducible Coxeter groups except the infinite dihedral groups.\end{rem*}

\ssect{Conjectural structure of reflection orders} The set of  quasiparabolic subsets is the largest natural class of subsets of $\Phi$  known to the author   for which conjecture  \ref{ss1.14}(a) seems plausible; for example, the set $\CB(\L)$ is not necessarily a lattice if   $\L$ is the complement, in a finite root system  $\Phi$ of type $A_{3}$,  of a rank one standard parabolic subset. 

 However, for at least some  quasiparabolic sets $\L$,  the coclosed subsets of $\L$ may not be the largest natural family of sets all of which have biclosed closure.  Say that a subset $\G$ of $\L=\Phi_{+}$ is unipodal if it has the following property:
if $\a\in \G$ and   $W'\in \CM_{\a}$  is a maximal dihedral refelction subgroup of $W$ containing $s_{\alpha}$, with canonical simple system  $\Pi_{W'}=\set{\b,\g}$ with respect to $(W,S)$, then either 
$\b\in \G$ or $\g\in \G$ (see \ref{ss5.3} for notation  and more details). It is easy to see that coclosed subsets of $\Phi_{+}$ are unipodal, so the following strengthens  the special case $\L=\Phi_{+}$ of Conjecture \ref{ss1.14}(c). 

\begin{conj*} If $\G\seq \Phi_{+}$ is unipodal, then $\ol{\G}$ is biclosed in $\Phi_{+}$.\end{conj*}

Some evidence for this conjecture will be given in another paper; in particular, it holds for finite Coxeter groups, by an argument using Bruhat order.
The conjecture  would also imply that an arbitrary  biclosed 
subset $\G$ of $\Phi_{+}$ is the directed union of the biclosed sets 
obtained as the closures of finite unipodal subsets of $\G$, and hence that the (conjectured)
 complete ortholattice   $\CB(\Phi_{+})$ is  an algebraic lattice (see  e.g. \cite{DavPr} for the definition).
 In conjunction with 
Conjecture \ref{ssinsect}, the above conjecture  would lead to a quite 
satisfactory
 description of 
reflection orders and their initial sections (for example, one could effectively 
compute  with them, in examples or in general,  by finite  ``approximations,'' in a 
similar manner as one may work with elements of  profinite groups). 

 Note however that if $(W,S)$ is an infinite dihedral Coxeter 
 system, then there are two ``exceptional'' quasiparabolic   subsets $\L$ of $\Phi$  such that 
 $\Phi=\L\dotcup -\L$  where  $\L\cap \Phi_{+}$ and $\L\cap \Phi_{-}$ are both 
 infinite; for these,
 $\CB(\L)$   is a complete ortholattice (as conjectured in \ref{ss1.14}) but not an algebraic lattice.

\ssect{Conjecture on biclosed subsets of quasi-positive systems}\label{ss:quasipos} In the terminology of \cite{DyRig}, a subset $\L$ of $\Phi$ is called a quasipositive system if   $\Phi=\L\dotcup -\L$. Thus, closed quasi-positive systems $\L$ are special    quasiparabolic subsets of  $\Phi$, and $\Phi_{+}$ is itself a closed  quasipositive system. The following   conjecture extends   Conjecture \ref{ssinsect}(b).
 \begin{conj*}  Let $\L$ be any closed quasipositive  system of $\Phi$. 
 Let $M$ be any maximal (under inclusion)  totally ordered subset of $\CB(\L)$. Then there is a total order $\preceq$ of $\L$ such that $M$ consists of all initial sections   of $\preceq$.\end{conj*}
 From the examples in  \ref{ss2.4}, one sees  that the conjecture does not extend as stated to general quasiparabolic subsets $\L$ of $\Phi$.
 
 \ssect{Conjecture on initial sections and Bruhat order} Define   a function \be \tau\colon \CP(\Phi_{+})\to \CP(W)\ee as follows: for any $\G\seq \Phi_{+}$, $\tau(\G)$ is the subset of $W$ consisting of all 
elements $w\in W$ such that there exist 
 $\a_{1},\ldots, \a_{n}\in \G$ with 
 $w=s_{\a_{1}}\cdots s_{\a_{n}}$ and 
 $0=l(1_{W})<l(1_{W}s_{\a_{1}})<\ldots <l(1_{W}s_{\a_{1}}\cdots s_{\a_{n}})$. 
 The motivation and natural context for the study of this function is in relation to    the twisted Bruhat orders of \cite{DyHS2}, which are  not discussed in this paper.
Instead,  $\tau$ is used  here to provide  another, quite different  description of the (conjectural) join in the poset $\CB(\Phi_{+})$ 
 
  \begin{conj*}    If $\G,\L\in \CB(\Phi_{+})$,   then $\mset{\a\in \Phi_{+}\mid s_{\a}\in\tau(\G\cup \L)}$  is the join (least upper bound)  of $\G$ and $\L$ in the poset $\CB(\Phi_{+})$\end{conj*}
 \begin{rem*}  The above conjecture  is   open even for finite Coxeter groups.
    As already mentioned, Conjecture \ref{ss1.14} is also open for finite Coxeter groups in the cases $2\leq \vert J\vert $; all other conjectures mentioned  are known to hold for finite Coxeter groups.  Some other results of this paper, such as Theorem \ref{ss1.8},  are  special cases  of general conjectures which are not stated here.  \end{rem*} 

\ssect{Aanalogous questions for  simplicial oriented geometries} 
There is  a natural ``convex geometric''  closure operator $d:\G\mapsto \bbR _{\geq 0}\G\cap \Phi$ on $\Phi$ taking a set of roots to the set of roots in its non-negative real span.  It  is shown   in Section \ref{S10} that many of the main results of the paper  hold  for $d$ just as for  $2$-closure. However,
while every $d$-biclosed subset of
$\Phi_{+}$ is an initial section of $\Phi_{+}$, it can be shown that there exist  finite rank $W$ for which  not every   initial section is $d$-biclosed, and thus the  $d$-analogue of Conjecture \ref{ssinsect}(a) fails.
This is unsurprising since $2$-closure  and initial sections are combinatorial in nature,
but $d$-closure is not (see \ref{ss10.3}).

 In Section \ref{S10} of this paper, it is shown that  the $d$-closure analogue of Theorem \ref{ss1.5} holds. We   do not know  whether, more generally,  the  analogue of
conjecture \ref{ss1.14}(a) in the special   case $\L=\Phi_{+}$, but for $d$-closure instead of $2$-closure, holds. However, 
 \cite[\S 5-6]{BjEZ} proves an   analogue of that conjecture in a different 
(and quite general) context, namely 
for the posets  of regions of a finite central simplicial hyperplane arrangement in 
a finite-dimensional real vector space (we consider only ``essential'' arrangements, namely those  for which the normals to the hyperplanes span the ambient vector space).  More generally, such an analogue holds for   simplicial oriented geometries
(which are   oriented matroids with additional properties; see op. cit. and \cite{BVSWZ} for  background).  It would be 
interesting to know if  Conjecture \ref{ss1.14}(c), for example, also has an 
analogue in that generality. There are (at least) two natural closure operators which one 
could use; the natural oriented matroid closure operator $d$ (see \cite{BF} or \cite[\S 6]{BjEZ}), and  an 
analogue of $2$-closure constructed from $d$ in a similar way as the $2$-closure on root 
systems is defined in terms of their geometric $d$-closure. In view of the results of this paper in the case of finite root systems, it  would  be particularly interesting to see how the $2$-closure behaves in simplicial geometries (it certainly does not have good properties for (possibly non-simplicial) oriented geometries in general).

Another interesting point for comparison of  \cite{BjEZ} with the results here   is the following. If the poset of regions of a finite central  hyperplane arrangement  in a real vector  space  is a lattice, then the base region
is simplicial by \cite[Theorem 3.1]{BjEZ}. However, we show in Section \ref{S10} of this paper that  the $d$-closure analogue of Theorem \ref{ss1.5}(a) holds even if the simple roots are linearly dependent (and  the fundamental chamber therefore not simplicial). These facts are  not contradictory, as 
the Coxeter group case corresponds to an infinite hyperplane arrangement and involves only a subset  of the ``regions'' (those in the $W$-orbit of the fundamental chamber i.e in the ``Tits
cone'').
\ssect{Questions on ``oriented geometry'' root systems} \lb{matroid}  A final  subtle point we wish to raise  concerning  the conjectures in this paper (such as  \ref{ss1.14}) is whether 
$2$-closure is the ``natural'' closure operator for use in their formulation.  It may well not be, as it does not induce an antiexchange closure on the positive roots in general (see Section \ref{S10}).

 It is  a formal fact (the proof of which is similar to that of \cite[Theorem 5.5]{BjEZ}
 and not given here)  that if Conjecture \ref{ss1.14}(a) holds for one closure relation, then  the closure operator $a$ on $\L$ in which the closed sets are  intersections of  elements of $\CB(\L)$ has $\CB(\L)$ as its $a$-biclosed sets and satisfies $\join X=a(\cup_{\G\in X}\G)$ for $X\seq \CB(\L)$.
It seems possible  that, if the conjectures hold for any closure operator on $\Phi$, there may be several different natural  closure operators $e$  on $\Phi$ for which they hold
(some evidence for this can be seen in Section \ref{S10}).

One  natural  candidate closure operator   we wish to informally describe  requires notions of  infinite oriented matroids, one definition of which   can be found in \cite{BF}.   We consider only a natural subclass,  which we shall henceforth just call oriented geometries, corresponding to oriented geometries in the case of  finite oriented  matroids (see also \cite[Exercise 3.13]{BVSWZ}). We shall not be more precise as we make only vague remarks below.  
  
   Take a $W\times \set{\pm 1}$-set   $\Psi:=T\times\set{\pm 1}$ 
   corresponding abstractly to the $W\times \set{\pm 1}$-set of roots of 
   $(W,S)$ in its standard root systems (see \ref{ss10.1a}). Consider   
   oriented geometry structures on this set, preserved by the $W$-action  
   of $(W,S)$, which restrict to the standard oriented geometry structure on 
   the roots of any dihedral reflection subgroup  (the standard structure 
   is the one obtained by transfer of structure  from any of their standard 
   root systems). Call such a structure an oriented geometry root system of 
   $W$;  any standard root system gives rise to one in this more general 
   sense.   Each oriented geometry root system gives (by definition in \cite{BF}) a closure operator on $\Psi$; these include analogues of  the 
   $d$-closures as previously considered, but one might  expect there 
   could be many more (not for 
   finite Coxeter groups, but at least for some  infinite non-affine Coxeter ones). The closure operator on $\Psi$ of interest (as a possibly 
   natural one for use in the  conjectures)  is the one which has as 
   closed sets the intersections of closed sets for these oriented geometry   root system closure operators. 
   It would also  be interesting (and possibly important in relation to the conjectures)  to know to  what extent  the results and conjectures of this paper
can be extended to the groupoids introduced in \ref{ss1.10} in this paper or the related and better studied groupoids such as as those in \cite{HY}, \cite{CH}, \cite{HV} (which is explictly concerned with weak order on Weyl groupoids)  and \cite{HowBr}.

\section{Example: Finite dihedral groups}\lb{S2}
In this section,  a number of the objects attached to Coxeter groups in Section \ref{S1} are  explicitly described in the case of a finite dihedral group $W$, and, at the end,  illustrated even more concretely  by example of the Weyl  group of type $A_{2}$. The reader is invited to consider the necessary changes for the infinite dihedral group. 
\subsection{Closed sets in finite dihedral root systems}  \lb{ss2.1} Throughout this section, we consider a finite  dihedral  Coxeter system $(W,S)$ of order $2m$ with simple roots $\Pi=\set{\a,\g}$. Then 
\[W=\set{1=1_{W}, s_{\a},s_{\g},s_{\a}s_{\g},s_{\g}s_{\a},s_{\a}s _{\g}s_{\a},\ldots,
w_{S}}\] where $w_{S}:= (s_{\a}s_{\g}s_{\a}\cdots)_{m}=(s_{\g}s_{\a}s_{\g}\cdots)_{m}$ is the longest element, and 
 \[\Phi_{+}=\mset{\a,s_{\a}(\g),s_{\a}s_{\g}(\a),\ldots}_{m}={}_{m}\set{...,s_{\g}s_{\a}(\g),s_{\g}(\a),{\g}}\] where each 
set on the previous line has $m$ elements, and elements    in corresponding positions in the two listed sets are equal. The order in which the elements are listed is one of the two  (mutually opposite) reflection orders of $\Phi_{+}$; they are the two possible orders   in which a ray sweeping around the origin (beginning and ending with a ray containing a negative root) would pass through the positive roots.
For example, in case $(W,S)$ is  of type $A_{2}$, we have $m=3$ and  $\Phi_{+}=\set{\a,\b,\g}$ where $\a=s_{\g}s_{\a}(\g)$,  $\b=\a+\g=s_{\a}(\g)=s_{\g}(\a)$ and $\g=s_{\a}s_{\g}(\a)$.

The closed subsets of $\Phi_{+}$ (in the general finite dihedral case) are the sets of positive roots  which may be obtained by deleting the first $j$ and last $k$  roots 
from the list 
\bee \a,s_{\a}(\g),s_{\a}s_{\g}(\a),\ldots, s_{\g}s_{\a}(\g),s_{\g}(\a),{\g}\eee  of elements of $\Phi_{+}$ in the above order, where $j,k\in \bbN$ with $j+k\leq n$.
The biclosed sets are the sets $\Phi_{w}$ for $w\in W$; they are the empty  set together with the subsets of $\Phi_{+}$ which 
are closed  and contain a simple root. 

\subsection{Stable subgroups and subsemilattices} \lb{ss2.2} In this subsection,   the stable subgroups and subsemilattices for the Galois connection associated to $R$ in \ref{ss1.9}  are described for the finite dihedral group $W$ (the analogous results for $R'$ can be obtained from this using the fact  that
 $\Phi'_{x}=\Phi_{xw_{S}}$).  Recall that $x  Rz $ if and only if  $z  (\Phi_{x })=\Phi_{x }$.
 Hence $1Rw$  and $wR1$ for all $w\in W$. Also, $wRw_{S}$ if and only if  $w=1$, and $w_{S}Rw $ if and only if  $w =1$.
 
   Let $z  ,x \in W\sm\set{ 1,w_{S}}$. We claim that $xRz$  if and only if  
  $z  =s_{\d}$ and $x =s_{\d} w_{S}$ for some $\d\in \Phi_{+}$ with $s_{\d}\neq w_{S}$. For suppose that $x  Rz $ holds. Then
 $z  $ must be a reflection (i.e. of odd length) since no non-identity rotation can fix any non-empty set  of positive roots (such as $\Phi_{x }$) setwise.  Suppose that $z  =s_{\d}$ where $\d\in \Phi_{+}$. There is some root $\e\in \Phi_{x }\cap \Pi$ since $x \neq 1$. It follows  that $\ol{\set{\e,s_{\d}(\e)}}\seq \Phi_{x }\seq \Phi'_{s_{\d}}$.
 But it is easy to check that the left hand side is just $\Phi'_{s_{\d}}$, so equality holds throughout. Hence $\Phi_{x }=\Phi'_{s_{\d}}=\Phi_{s_{\d}w_{S}}$ and $x =s_{\d}w_{S}$. On the other hand, $s_{\d}(\Phi_{s_{\d}w_{S}})=
 s_{\d}(\Phi'_{s_{\d}})=\Phi'_{s_{\d}}=\Phi_{s_{\d}w_{S}}$, and the claim is proved.
 
 It now follows from the definitions that the pairs $(G,G^{*})$ of  stable subgroup $G$ of $W$ and corresponding stable meet subsemilattice $G^{*}$ of $W$ are
 $(G=W,G^{*}=\set{1})$, $(G=\set{1_{W}},G^{*}=W)$ and $(G=\set{1,s_{\d}},G^{*}=\set{1,s_{\d}w_{S}})$ for $\d\in \Phi_{+}$ with $s_{\d}\neq w_{S}$. There are thus  $m+1$ stable pairs if $m$ is odd, and $m+2$ if $m$ is even; the map $w\mapsto \set{w}^{*}$ gives an bijection  from the set of elements $w\in W$ with $w^{2}=1_{W}$ to the set of stable subsemilattices. The lattice of stable subgroups (resp., stable subsemilattices) is a  poset  with $W$ as the maximum element, $\set{1_{W}}$ as the minimum element and all other elements pairwise incomparable.

\subsection{Standard parabolic weak orders} \lb{ss2.3}
This subsection describes the standard parabolic weak orders
$\CL_{J}$, where $J\seq S$, for the finite dihedral group $W$.  

First, the description of weak order is well known.  There is a maximum element
$w_{S}$, minimum element $1$, and exactly two maximal chains from $1$ to $w_{S}$, namely \bee 1<s_{\a}<s_{\a}s_{\g}<s_{\a}s_{\g}s_{\a}<\ldots <w_{S}
\text{ \rm  and }1<s_{\g}<s_{\g} s_{\a }<s_{\g}s_{\a}s_{\g}<\ldots <w_{S},\eee both of length $m$ (i.e. with $m+1$ elements).  The poset $\CL_{\eset}$ is order isomorphic to $(W,\leq)$ under 
the map $w\mapsto \Phi_{w}$.

The poset $\CL_J$ for $J=\set{s_{\a}}$ may be described as follows (the description for $J=\set{s_{\g}}$ is obtained  by symmetry). 
In this case, the poset $\CL_{J}$ has a maximum element $\top:=\Phi_{+}\cup\set{-\a}$ and a minimum element $\bot=\eset$. The group $\set{1,s_{\a}}$ acts by order automorphisms of the poset by $(w,\G)\mapsto w(\G)$.
There are $2n+4$ elements of $\CL_{J}$,  namely $\eset$,  $\Phi_{+}\sm\set{\a}$, $\set{\a,-\a}$, $\top$ and, for each $w\in W$ with $\a\in \Phi_{w}$, $\Phi_{w}$ and $s_{\a}(\Phi_{w})$.
There are   $6$ maximal chains from $\bot$ to $\top$. Three of the maximal chains are \bee \eset<\Phi_{s_{\a}}=\set{\a}<\Phi_{s_{\a}s_{\g}}<\
\Phi_{s_{\a}s_{\g}s_{\a}}<\ldots <\Phi_{w_{S}}=\Phi_{+}<\top\eee (of length $m+1$), $\eset<\set{\a} <\set{\a,-\a}<\top$  (of length $3$)
and $\eset<\Phi_{+}\sm\set{\a}<\Phi_{+}<\top$ (also of length $3$); the other three maximal chains are obtained by acting on these by $s_{\a}$.

Finally, we describe $\CL_{S}$. The group $W$ acts on $\CL(S)$ as a group of order automorphisms. The elements of $\CL_{S}$ are the minimum element $\bot=\eset$, maximal element $\top=\Phi$ and  elements
$\Psi_{+}$, $\Psi_{+}\sm\set{\d}$, $\Psi_{+}\cup\set{-\d'}$ where $\Psi_{+}$ runs over positive systems of $\Phi$ and for each $\Psi$, $\d$ and $\d'$ run over the simple roots of $\Psi_{+}$.  There are $6m+2$ elements of $\CL_{S}$ in all.
There are $8m$ maximal chains from $\eset$ to $\Phi$, all of length
$4$. They are exactly the chains of the form
$\eset<\Psi_{+}\sm\set{\d}<\Psi_{+}<\Psi_{+}\cup\set{-\d'}<\Phi$ for $\Psi$, $\d$, $\d'$  as above.   \begin{rem*} It will be shown elsewhere in conjunction with a proof of Conjecture \ref{ssinsect} for affine Weyl groups  that  if $W$ is any  finite Coxeter group, then  $\CL_{S}$ is an ortholattice with maximum element $\Phi$, minimum element $\eset$, and 
 in which  every maximal chain from $\eset$ to $\Phi$ has length $2\vert S\vert$. However, it has  not been  checked that the join is as conjectured in \ref{ss1.14}(a). \end{rem*} 
\subsection{The type $A_{2}$ case} \lb{ss2.4} In this subsection, we take $(W,S)$   of type $A_{2}$ with notation as in \ref{ss2.1} i.e. $\Pi=\set{\a,\g}$ and  $\Phi_{+}=\set{\a,\b,\g}$ where  $\b=\a+\g$.

The Hasse diagrams of the stable subgroups ordered by inclusion and 
corresponding stable subsemilattices ordered by reverse inclusion are as follows:
 
 \bee 
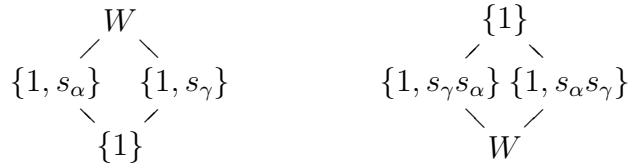
\xymatrix@!0{ &W\ar@{-}[dl]\ar@{-}[dr]&   &&&& 
  &{\set{1}}\ar@{-}[dl]\ar@{-}[dr]\\
 {\set{1,s_{\a}}}&&{\set{1,s_{\g}}} && &&
 {\set{1,s_{\g}s_{\a}}}&&{\set{1,s_{\a}s_{\g}}}\\
 &{\set{1}}\ar@{-}[ul]\ar@{-}[ur]&  &&&& &W\ar@{-}[ul]\ar@{-}[ur]& }\eee
 
The Hasse diagram of the parabolic order  $\CL_{J}$ where $J=\set{s_{\a}}$ is
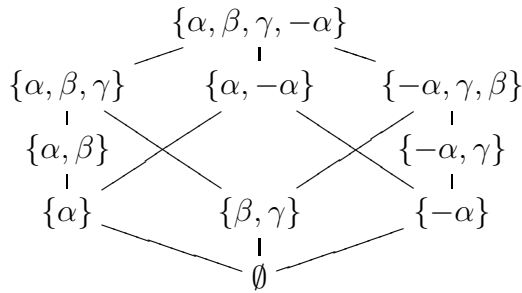
\bee \xymatrix@!0{&&&\set{{\a,\b,\g,-\a}}\ar@{-}[dlll]\ar@{-}[d] \ar@{-}[drrr]&&&\\
{\set{\a,\b,\g}}\ar@{-}[d]&&&{\set{\a,-\a}}\ar@{-}[ddlll]\ar@{-}[ddrrr]&&&{\set{-\a,\g,\b}}\ar@{-}[d]\\
{\set{\a,\b}}\ar@{-}[d]&&&&&&{\set{-\a,\g}}\ar@{-}[d]\\
{\set{\a}}&&&{\set{\b,\g}}\ar@{-}[uulll]\ar@{-}[uurrr]&&&{\set{-\a}}\\
&&&\eset\ar@{-}[ulll]\ar@{-}[u] \ar@{-}[urrr]&&&}
\eee
Note that $\CL_{\set{s_{\a}}}$ is not graded;  maximal chains may have different cardinalities.

 The Hasse diagram of the  parabolic order  $\CL_{S}$  is of the form
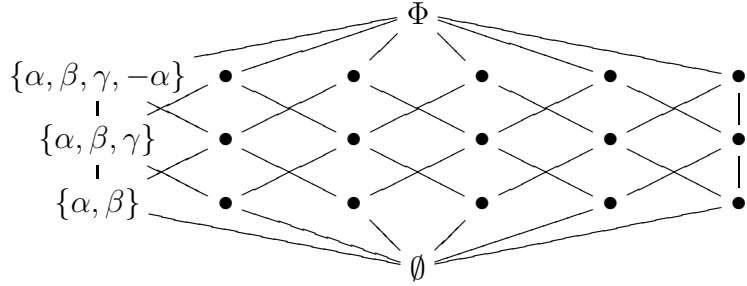
\bee \xymatrix@!0{&&&&&\Phi\ar@{-}[dlllll]\ar@{-}[drrrrr]\ar@{-}[drrr]\ar@{-}[dlll]\ar@{-}[dr]\ar@{-}[dl]
&&&&&\\
{\set{\a,\b,\g,-\a}}\ar@{-}[d]\ar@{-}[drr]&&\bullet\ar@{-}[dll]\ar@{-}[drr]
&&\bullet\ar@{-}[dll]\ar@{-}[drr]&&\bullet\ar@{-}[dll]\ar@{-}[drr]
&&\bullet\ar@{-}[dll]\ar@{-}[drr]&&\bullet\ar@{-}[d]\ar@{-}[dll]\\
{\set{\a,\b,\g}}\ar@{-}[d]\ar@{-}[drr]&&\bullet\ar@{-}[dll]\ar@{-}[drr]
&&\bullet\ar@{-}[dll]\ar@{-}[drr]&&\bullet\ar@{-}[dll]\ar@{-}[drr]
&&\bullet\ar@{-}[dll]\ar@{-}[drr]&&\bullet\ar@{-}[d]\ar@{-}[dll]\\
{\set{\a,\b}}&&\bullet&&\bullet&&\bullet&&\bullet&&\bullet\\
&&&&&\eset\ar@{-}[ulllll]\ar@{-}[ulll]\ar@{-}[ul]\ar@{-}[ur]\ar@{-}[urrr]
\ar@{-}[urrrrr]&&&}\eee
where we have only explicitly  indicated the elements of one maximal chain.

  \section{Preliminaries} \lb{S3}
   \subsection{Finite biclosed subsets of the positive roots} \lb{ss3.1} Recall the  terminology concerning biclosed sets  from \ref{ss:term}.
 
\begin{lem*} \begin{num}\item  For $\G\seq \Phi_{+}$ and $x\in W$, set $x\cdot \G:=\bigl (\Phi_{x}\sm x(-\G)\bigr)\cup \bigl(x(\G)\sm(-\Phi_{x})\bigr)$.
Then  $(w,\G)\mapsto w\cdot \G$  gives an action of the group $W$ on the power set $\CP(\Phi_{+})$.
 \item $x\cdot \Phi_{y}=\Phi_{xy}$ and $x\cdot \Phi'_{y}=\Phi'_{xy}$ for $x,y\in W$.
\item  If $\G\subseteq \Phi_{+}$ is biclosed, then
 $w\cdot \G$ is biclosed  for all $w\in W$.
 \item A finite subset $\G$ of $\Phi_{+}$ is biclosed if and only if 
it is of the form $\G=\Phi_{w}$ for some $w\in W$.
\item  $\Phi_{xy}=\bigl(\Phi_{x}\sm(-x(\Phi_{y}))\bigr)\dotcup \bigl(x(\Phi_{y})\sm (-\Phi_{x})\bigr)$ where $\Phi_{x}\cap x(\Phi_{y})=\eset$, for any $x,y\in W$,. \item For $x,y\in W$, we have  $x\leq xy$ if and only if  $l(xy)=l(x) +l(y)$ if and only if   $\Phi_{x^{-1}}\cap \Phi_{y}=\emptyset$ if and only if 
$x(\Phi_{y})\seq \Phi_{+}$ if and only if 
$\Phi_{xy}=\Phi_{x}\dotcup x(\Phi_{y})$ if and only if  $\Phi_{x}\seq \Phi_{xy}$ if and only if  $x(\Phi_{y})\seq \Phi_{xy}$.\end{num}\end{lem*}

 \begin{proof} There is a $W$-action $(w,A)\mapsto w\cdot A=N(w)+wAw^{-1}$ on $\CP(T)$  where \be \lb{eqcoc}  N(w)=N_{(W,S)}(w):=\mset{s_{\a}\mid \a\in \Phi_{w}}=\mset{t\in T\mid l(tw)<l(w)}\ee and $+$ denotes symmetric difference. This action was used  in \cite{DyHS1} and  \cite{DyHS2}, for instance; the fact that the formula gives an action follows from the cocycle property
 \be N(xy)=N(x)+xN(y)\ee for $x,y\in W$.
   The $W$-action on $\CP(\Phi_{+})$  in (a)  is easily seen to be obtained from this action by transfer of structure using the bijection $\a\mapsto s_{\a}\colon \Phi_{+}\rightarrow T$. For further discussion of this action and of its geometric interpretation, see \cite[1.1--1.2]{DyQuo}. 
 
  The formula in (a) immediately shows that $\Phi_{y}=y\cdot \eset$  and then
  $\Phi_{xy}=(xy)\cdot \eset=x\cdot(y\cdot \eset)=x\cdot \Phi_{y}$, proving the first part of (b). The second part of (b) follows since the formula in (a) implies that
  $x\cdot (\Phi_{+}\sm \G)=\Phi_{+}\sm (x\cdot \G)$.  Part (c) is proved by  reducing to the easily checked case of dihedral groups  by considering the intersections of $\G$ with the maximal dihedral reflection subgroups (see \ref{ss5.3}) of $W$;  see \cite[Proposition 2.6]{Edg}.
  Part (d) is proved in \cite{Pil} for the standard reflection representation of \cite{Bour} or \cite{Hum} by a straightforward  modification of a well known argument for finite Weyl groups. Exactly the same argument  as in \cite{Pil} applies to the class of root systems we consider here to establish (d).
 Another proof of (d). is as follows. Note that since $\emptyset$ is clearly biclosed, 
 so is $\Phi_{x}=x\cdot \eset$  for $x\in W$ by (b). The reverse implication will be proved 
 using the easily checked fact that for   any non-empty biclosed set  $\Delta'$,
 a root    $\alpha\in \D'$ with $s_{\a}$ of minimal length must be simple; then 
 $\vert s_{\a}\cdot \D'\vert+1=\vert\D'\vert$. Let  $\Gamma$ be any finite biclosed set and chose  $\Delta$ of minimal cardinality in the orbit
 $W\cdot \Gamma$.  If $\D\neq \eset$, applying  the above with $\D'=\D$ gives  a contradiction to minimality of $\vert \D\vert$. Hence  $\eset\in W\cdot \G$, $\Gamma\in W\cdot \emptyset$, and,  say $\Gamma=x\cdot \eset=\Phi_{x}$ as required.
 
For (e)--(f), note that the definitions give $\Phi_{x^{-1}}=-x^{-1}(\Phi_{x})$. Part (e) is  a straightforward consequence of (a)--(b)
 and the fact that $x^{-1}(\Phi_{x}\cap x(\Phi_{y}))=x^{-1}(\Phi_{x})\cap \Phi_{y}=-\Phi_{x^{-1}}\cap \Phi_{y}\seq \Phi_{+}\cap \Phi_{-}=\eset$.
 Then (f) follows easily from (e)   on recalling that $l(z)=\vert \Phi_{z}\vert$ for any $z\in W$.\end{proof}

  \ssect{A more general order isomorphism} \lb{ss1.6a} The following Lemma will be used  in Section \ref{S9} (compare also Corollary \ref{ss1.6}).
   \begin{lem*} \begin{num} \item For any $x\in W$, the map $\G\mapsto x\cdot \G$ induces an order isomorphism 
\bee 
\mset{\G\in \CB(\Phi_{+})\mid \G\cap \Phi_{x^{-1}}=\eset}\xrightarrow{\cong}
\mset{\D\in \CB(\Phi_{+})\mid\Phi_{x}\seq \D}.\eee
\item If a non-empty subset $X$ of $\CB(\Phi_{+})$ is such that
$\G\cap \Phi_{x^{-1}}=\eset$ for all $\G\in X$, and $X$ has a join $\L=\join X$ in
$\CB(\L)$, then $\L\cap \Phi_{x^{-1}}=\eset$. \end{num}
\end{lem*}
\begin{proof} 
The  map in (a) is the restriction of a similar order isomorphism with $\CB$ replaced by $\CP$. Explicitly, the inverse bijections are $\G\mapsto \D=\Phi_{x}\cup x(\G)$ and $\D\mapsto \G=x^{-1}(\G\sm \Phi_{x})$. One may put $\CB$ in place of $\CP$ since $\CB(\Phi_{+})\seq\CP(\Phi_{+})$ is $W$-stable.

For (b), one notes that for all $\G\in X$, one has $\G\seq \Phi'_{x^{-1}}\in \CB(\Phi_{+})$, so $\Phi'_{x^{-1}}$ is an upper bound for $X$ and therefore $\join X\seq \Phi'_{x^{-1}}$ i.e. $\L\cap \Phi_{x^{-1}}=\eset$.

\end{proof}

  \subsection{Trivial properties of closure} \lb{ss3.2} The following   assorted simple  facts are stated  for future reference,  omitting  the proofs. 
  \begin{lem*} Let $\G,\D\seq\Phi$ and  $w\in W$.    \begin{num}
  \item Recursively define  $\G_{0}:=\G$ and $\G_{n+1}=\cup_{\a,\b\in \G_{n}}\ol{\set{\a,\b}}$ for $n\in \bbN$.
  Then $\G_{0}\seq \G_{1}\seq \G_{2}\seq\ldots$ and $\ol{\G}=\cup_{n\in \bbN}\G_{n}$.  \item $\ol{w(\G)}=w(\ol{\G})$.
   \item $\ol{\G\cup\D}\sreq \ol{\G}\cup\D$ and  $\ol{\ol{\G}\cup\D}=\ol{\G\cup\D}$.
  \end{num}
  \end{lem*}

 \section{Closure and joins}\lb{S4}
  Denote the join of a family of elements $\set{x_{i}}$ of $W$ in weak order as $\join_{i}x_{i}$ when it exists (which it may not). Similarly, denote the meet as $\meet_{i}x_{i}$ when it exists.   Also write $\join X$   and $\meet X$ for the join and meet of $X\seq W$ when they exist. 
  This section will  prove   Theorem \ref{ss1.5}(a), describing joins in weak order in terms of $2$-closure, and  show   that $(W,\leq)$ is a complete meet semilattice.

 \subsection{Joins in cosets of rank two standard parabolic subgroups} \lb{ss4.1} The  proof of Theorem \ref{ss1.5}(a) begins with the following observation.  
 \begin{lem*} Let $x\in W$ and $\alpha,\beta\in \Pi$ with $x<y:=xs_{\a}$ and $x<z:=xs_{\b}$.
 Then $y$ and $z$ have an upper bound in $(W,\leq)$ if and only if  the standard parabolic subgroup $W':=\mpair{s_{\alpha},s_{\beta}}$ of $W$ is finite if and only if  $\ol{\set{{\a},{\b}}}$ is finite.  In that case, let $w$ denote the longest element of $W'$. Then
 \begin{num}\item $y\vee z=xw$
 \item $\Phi_{xw}=\Phi_{x}\dotcup  x(\Phi_{w}) =\Phi_{x}\dotcup\overline{\set{x(\a),x(\b)}}$ \item $\Phi_{y\vee  z}=\Phi_{x}\dotcup\ol{(\Phi_{y}\sm \Phi_{x})\cup (\Phi_{z}\sm \Phi_{x})}=\ol{\Phi_{y}\cup \Phi_{z}}$.
\end{num}
 \end{lem*}
 \begin{proof} 
 Since $x<xs_{\a}$ and $x<xs_{\b}$, it follows that $x$ is the (unique) element of minimal length
 in the coset $xW'$. Hence \bee l(xw')=l(x)+l(w'),\quad \Phi_{xw'}=\Phi_{x}\dotcup  x(\Phi_{w'}),\quad x\leq xw'\eee for all $w'\in W$. In particular, $\Phi_{y}=\Phi_{x}\dotcup \set{x(\a)}$ and 
  $\Phi_{z}=\Phi_{x}\dotcup \set{x(\b)}$.
 Also note that the set of positive roots of $W'$ is $\ol{\set{\a,\b}}$ so $W'$ is finite if and only if  
 $\ol{\set{\a,\b}}$ is finite, and in that case the longest element $w$ of $W'$ satisfies $\Phi_{w}=\ol{\set{\a,\b}}$ and \bee \Phi_{xw}=\Phi_{x}\dotcup 
 x\ol{\set{\a,\b}}=\Phi_{x}\dotcup  \ol{\set{x(\a),x(\b)}}.\eee 
 Finally, $u\in W$ is an upper bound of $y$ and $z$ if and only if  $\Phi_{y}\seq \Phi_{u}$ and $\Phi_{z}\seq\Phi_{u}$. This holds if and only if  $\Phi_{y}\cup\Phi_{z}\seq \Phi_{u}$ or, equivalently, if and only if  $\ol{\Phi_{y}\cup\Phi_{z}}\subseteq \Phi_{u}$, since $\Phi_{u}$ is closed.
 
 Now suppose that an upper bound of $y$ and $z$ exists, and let $u$ be any such upper bound.
 From above, it follows  that $\Phi_{x}\dotcup\ol{\set{x(\a),x(\b)}}\subseteq \ol{\Phi_{y}\cup\Phi_{z}}\seq\Phi_{u}$. In particular, $ \ol{\set{x(\a),x(\b)}}$ is finite, so $W'$ is finite, $\Phi_{xw}\subseteq \Phi_{u}$ and $xw\leq u$.  Hence if an upper bound for $y$ and $z$ exists, then $W'$ is finite, 
 and any upper bound $u$ satisfies $u\geq xw$. On the other hand, if $W'$ is finite, then the  above proves that
 $\Phi_{y},\Phi_{z}\subseteq \Phi_{xw}$, so  $xw$ is an upper bound of $y$ and $z$. The first part of this paragraph with $u=xw$  shows  that   $y\vee z=xw$ and
 \bee \Phi_{xw}=\Phi_{x}\dotcup\ol{\set{x(\a),x(\b)}}\subseteq \ol{\Phi_{y}\cup\Phi_{z}}\seq\Phi_{xw}.\eee This proves the first assertion of the Lemma, and its parts (a)--(b). 
 Part (c) follows from (a)--(b) and what has just been proved, 
 using $\Phi_{y}\sm \Phi_{x}=\set{x(\a)}$ and
  $\Phi_{z}\sm \Phi_{x}= \set{x(\b)}$.
 \end{proof}
 \subsection{Joins in weak order}\lb{ss4.2} The following proof is quite similar to that of \cite[Lemma 2.1]{BjEZ}, taking account of  the extra structure  of concern  here.
 \begin{prop*} Suppose that $x,y,z,u\in W$ with $x\leq y\leq u$ and $x\leq z\leq u$.
 Then $y\vee z$ exists and
$\Phi_{y\vee  z}=\Phi_{x}\dotcup\ol{(\Phi_{y}\sm \Phi_{x})\cup (\Phi_{z}\sm \Phi_{x})}=\ol{\Phi_{y}\cup \Phi_{z}}$.
\end{prop*}
\begin{rem*}  The more complicated statement here as compared to the case $\vert X\vert=2$ of \ref{ss1.5}(a)  is only  to facilitate the proof. The two statements are   equivalent,  using Corollary \ref{ss1.6}(a).
It can be shown that Conjecture \ref{ss1.14} implies, for example, that for $\L,\G,\D\in \CB(\Phi_{+})$ with
$\L\seq \G\cap \D$, one has $\ol{\G\cup \D}=\L\cup\ol{(\G\sm \L)\cup (\D\sm \L)}$.
 \end{rem*}
  \begin{proof} Observe that the union $ \Phi_{x}\cup\ol{(\Phi_{y}\sm \Phi_{x})\cup (\Phi_{z}\sm \Phi_{x})}$ is one of disjoint sets since $\ol{(\Phi_{y}\sm \Phi_{x})\cup (\Phi_{z}\sm \Phi_{x})}
\subseteq\ol{ \Phi_{x}'}= \Phi_{x}'$. Note that if  $x=y$, then $y\vee z=x\vee z=z$ and the result is trivial since $\Phi_{z}$ is closed. Similarly, the result is trivial if $x=z$. In particular, the result holds if $l(x)=l(u)$
  (in which case $x=y=z=u$). The Proposition  will be proved by induction on $N:=l(u)-l(x)$.
  Assume inductively that $N\in \bbN_{>0}$ and the assertion of the Proposition holds for all $x,y,z,u$ satisfying the hypotheses of  the Proposition with $l(u)-l(x)<N$. Let $x,y,z,u$ satisfy the hypotheses  with $l(u)-l(x)=N$. As above,  assume  without loss of generality that $x\neq y$ and $x\neq z$.
  
Fix a simple reflection $r$ satisfying  $x<x':=xr\leq z$. Consider the following hypothesis \eqref{H}:
\bee \tag{H}\label{H}\text{ $y':=y\vee x'$ exists and $\Phi_{y'}=\Phi_{x}\dotcup\ol{(\Phi_{y}\sm \Phi_{x})\cup (\Phi_{x'}\sm \Phi_{x})}=\ol{\Phi_{y}\cup \Phi_{x'}}$.}\eee 
We claim that \eqref{H} implies the conclusion \eqref{A} of the Proposition:
\bee\tag{A}\label{A} \text{$y\vee z$ exists and
$\Phi_{y\vee  z}=\Phi_{x}\dotcup\ol{(\Phi_{y}\sm \Phi_{x})\cup (\Phi_{z}\sm \Phi_{x})}=\ol{\Phi_{y}\cup \Phi_{z}}$.}\eee

Assume for the proof of the claim that \eqref{H} holds. Since $y\leq u$ and $x'\leq z\leq u$, it follows that 
$y'=y\vee x' \leq u$. Hence there is  a diagram indicating some of the order relations  in $(W,\leq)$ as follows:

   \bee \xymatrix@!0{&&{u}&\\
   &{y'}\ar@{-}[ur]&&{z}\ar@{-}[ul]\\
   {y}\ar@{-}[ur]&&{x'}\ar@{-}[ur]\ar@{-}[ul]&\\
   &{x}\ar@{-}[ur]\ar@{-}[ul]&&}\eee 
  Compute   
\be\label{comp}   \begin{split} \G&:=\ol{\Phi_{y}\cup \Phi_{z}}\\
   &\,\sreq\,\ol{(\Phi_{y}\sm \Phi_{x})\cup (\Phi_{z}\sm \Phi_{x})}\cup \Phi_{x}
   \\&= \ol{(\Phi_{y}\sm \Phi_{x})\cup (\Phi_{x'}\sm \Phi_{x})\cup (\Phi_{z}\sm \Phi_{x'})}\cup \Phi_{x}
    \\&= \ol{(\Phi_{y'}\sm \Phi_{x'})\cup (\Phi_{x'}\sm \Phi_{x})\cup (\Phi_{z}\sm \Phi_{x'})}\cup \Phi_{x}\\
    &\,\sreq\, \ol{(\Phi_{y'}\sm \Phi_{x'})\cup (\Phi_{z}\sm \Phi_{x'})} \cup (\Phi_{x'}\sm \Phi_{x})\cup \Phi_{x}\\
   & =\ol{(\Phi_{y'}\sm \Phi_{x'})\cup (\Phi_{z}\sm \Phi_{x'})} \cup \Phi_{x'} =:\D\end{split} \ee
    where we use \eqref{H} and Lemma \ref{ss3.2}(c) (resp., Lemma  \ref{ss3.2}(c)) to get the fourth (resp., second and fifth) line.
    Since $l(u)-l(x')=N-1<N$, the inductive hypothesis implies that $y'\vee z$ exists
    and $\D=\Phi_{y'\vee z}$. Now $\D\supseteq \Phi_{y'}\cup \Phi_{z}\,\sreq\,\Phi_{y}\cup \Phi_{z}$.  Hence $\D=\ol{\D}\,\sreq \,\ol{\Phi_{y}\cup\Phi_{z}}=\G$ and so the containments in \eqref{comp} are 
    all equalities. Clearly,
    $y'\vee z$ is an upper bound for $\set{y,z}$. On the other hand, if $v$ is any upper bound for 
  $\set{y,z}$, then $\Phi_{y}\cup\Phi_{z}\,\seq\, \Phi_{v}$ so $\Phi_{y'\vee z}=\G=\,\ol{\Phi_{y}\cup\Phi_{z}}\,\seq\, \Phi_{v}$ which implies $y'\vee z\leq v$. This shows that $y\vee z$ exists and in fact that  $y\vee z=y'\vee z$.  From \eqref{comp} and $\G=\D$, it follows that \eqref{A} holds.
  This proves the claim that \eqref{H} implies \eqref{A}.
  
  Hence we are reduced to proving \eqref{H}.
   Fix a simple reflection $s$ with $x<xs\leq y$. Consider the following hypothesis \eqref{H'}:
\bee \tag{H$'$}\label{H'}\text{ $xs\vee xr$ exists and $\Phi_{xs\vee xr}=\Phi_{x}\dotcup\ol{(\Phi_{xr}\sm \Phi_{x})\cup (\Phi_{xs}\sm \Phi_{x})}=\ol{\Phi_{xr}\cup \Phi_{xs}}$.}\eee 
 Since  $x<xr\leq u$, $x<y\leq u$ and $l(u)-l(x)\leq N$, replacing $(x,y,z,u,r)$ by
 $(x,xr,y,u,s)$ in the above proof that \eqref{H} implies \eqref{A} shows that
 \eqref{H'} implies \eqref{H}.  Hence it remains  only to prove \eqref{H'}. But \eqref{H'} follows from
 Lemma \ref{ss4.1}, and  so the proof of the Proposition is complete.
   \end{proof}
   \ssect{Meets from joins}\lb{ss4.3} The  following well known simple    facts are used in the argument to show that the weak order on $W$ is a complete meet semilattice.
   \begin{lem*}  Let $(\L,\preceq)$ be a  poset with a minimum element, denoted $\bot$,
   such that for every $x\in \L$, $ \mset{w\in W\mid w\preceq x}$ is finite.
   Assume that any two elements of $\L$ with an upper bound  have a least upper bound. 
   Then $\L$ is a complete meet semi-lattice i.e.  any non-empty subset of $\L$ has a greatest lower bound. Further, any non-empty subset with an upper bound has a least upper bound. \end{lem*}
   \begin{proof}  The assumptions imply (by induction on $\vert B\vert$) that the join $\join B$ exists for any finite non-empty subset $B$ of $\L$ with an upper bound. Since any non-empty subset with an upper bound is finite, it has a least upper bound, proving the last assertion.
   Consider now any non-empty subset $A$ of $\L$.
  We have to show  the  greatest lower bound $\wedge A$ exists.
   Consider the  set $B$ of lower bounds of $A$. For any $a\in A$, we have $b\leq a$ for all $b\in B$, so $B$ is finite, non-empty (it contains $\bot$) and bounded above. It follows $B$ has a join $b:=\join B$.
   For any $a\in A$, we have that $a$ is an upper bound of $B$ and so $b\leq a$.
    Hence $b$ is a lower bound of $A$ i.e. $b\in B$. This implies that $b$ is the maximum element of $B$ i.e. $b=\meet A$ as required. \end{proof}

   \subsection{The semi-lattice property of weak order} \lb{ss4.4} The Corollary below  summarizes the  main results of this section.    \begin{cor*}\begin{num}
   \item $(W,\leq)$ is a complete meet   semilattice.   \item If a   subset  of $W$ has an upper bound in $W$, then it has a join in $(W,\leq)$ given by the formula in $\text{\rm Theorem \ref{ss1.5}(a)}$.
   \item If $W$ is finite, then the meet in $W$ is given by the formula in  
   $\text{\rm Theorem \ref{ss1.5}(b)}$.\end{num} \end{cor*}
  \begin{proof} The join of a family of two elements (with an upper bound) is given by \ref{ss1.5}(a).
  Then the formula in \ref{ss1.5}(a) follows for finite subsets $X$  by  induction on $\vert X\vert$. 
  The formula in \ref{ss1.5}(a) then applies to any non-empty subset $X$ with an upper bound, since such a set $X$ is finite, proving (b).  Part (a) now  follows from Lemma \ref{ss4.3}.
   To prove (c),  recall that (assuming $W$ finite)  the longest  element $w_{S}$ of $W$ satisfies   $w_{S}\Phi_{+}=-\Phi_{+}$, $w_{S}^{2}=1_{W}$ and  $\Phi_{x}'=\Phi_{xw_{S}}$ for  all $x\in W$. Further, the map $x\mapsto xw_{S}$ is an order-reversing bijection of $W$ with itself.  (See for instance \cite[2.3.1, 2.3.2, 3.2.2]{BjBr} for these well-known facts).
      Now Theorem \ref{ss1.5}(b) for finite $W$ follows easily from Theorem \ref{ss1.5}(a) using these facts.
   \end{proof}
   
   \begin{rem*} If $X=\set{x,y}\seq W$ and  $\join X$ exists, a similar (essentially, dual) argument to the proof of Proposition  \ref{ss4.2} shows that the meet $\meet X$ is given by the formula in \ref{ss1.5}(b); in particular, this argument can be extended to give another proof of Corollary \ref{ss4.4}(c). However, $X$ need not have a join, so one needs a different argument to prove \ref{ss1.5}(b) in general. \end{rem*}

\ssect{Proof of Corollary \ref{ss1.6}}
 Corollary \ref{ss1.6}(a) is an easy consequence of the definitions (cf. \cite[Proposition 3.1.6]{BjBr}). Using Lemma \ref{ss3.1}(f),  \ref{ss1.6}(b) says that the domain $D_{x}$ of the order isomorphism in (a) is closed under taking those joins which exist in  $W$. 
This  holds since $D_{x}:=\mset{u\in W\mid  \Phi_{u}\seq \Phi'_{x^{-1}}}$, so if $U\seq D_{x}$,
then \bee \Phi_{y}=\ol{\cup_{u\in U}\Phi_{u}}\seq\ol{\cup_{u\in U}\Phi'_{x^{-1}}}=\Phi'_{x^{-1}}\eee
by Theorem \ref{ss1.5}(a) and therefore $y\in D_{x}$.

\section{Closure and Chevalley-Bruhat order} \lb{S5}
This section proves Lemma \ref{ss1.7} after giving requisite background on the Bruhat graph and reflection subgroups. \subsection{Bruhat graph}\lb{ss5.1} Define an edge-labelled, directed graph called  the Bruhat graph $\Omega=\Omega_{(W,S)}$ of $(W,S)$ as follows (see \cite{DyBru}).
The vertex set  of $\Omega$ is $W$. There  is an edge $(x,y)$ from $x$ to $y$ for each
$x,y\in W$ such that $l(x)<l(y)$ and  $y=s_{\alpha}{x}$  for some (necessarily unique) $\alpha\in \Phi_{+}$. Let $E=E_{(W,S)}$ denote the set of edges of $\Omega$.
Endow $\Omega$ with an edge labelling by attaching to the edge $(x,y)\in E$ as above the label
$L_{x,y}:=\a$ where $\a\in \Phi_{+}$ with $y=s_{\a}x$.
For any subset $V$ of $W$,  let $\Omega(V)$ denote the full edge-labelled subgraph of 
$\Omega$ on vertex set $V$ (i.e. with edge set $E\cap (V\times V)$).

\subsection{Bruhat order}\lb{ss5.2} The Chevalley-Bruhat order,  which is denoted  here  either as $\leq_{\emptyset}$ to distinguish it from weak order or as $\leq_{W,\emptyset}$ to indicate dependence on $W$,
 is the partial order on $W$ defined by the condition that $x\leq_{\eset} y$ if and only if  there is $n\in \bbN$ and  a path of length $n$
 from $x$ to $y$ in $E$ i.e. there exist $x=x_{0},x_{1},\ldots, x_{n}=y$ in $W$ such that
 $(x_{i-1},x_{i})\in E$ for $i=1\ldots, n$). Write $[x,y]_{\eset}= [x,y]_{W,\eset}:=\mset{z\in W\mid x\leq_{\eset}z\leq_{\eset}y}$.
 
 Note that  the Hasse diagram of Chevalley-Bruhat order, when regarded as directed graph with edges $(x,y)$ for $x,y\in W$ with $x<_{\eset}y$ and $l(y)=l(x)+1$, is a subgraph of $\Omega$. Then $\Phi_{x,1}$ (resp., $\Phi_{x,-1}$) as defined in \ref{ss1.7} is the set of labels in $\Omega$ of edges of the directed Hasse diagram with $x$ as initial (resp., terminal) vertex, corresponding to the elements which cover (resp., which are covered by) $x$ in the order $\leq_{\eset}$.

\subsection{Maximal dihedral reflection subgroups}\lb{ss5.3} From  \cite{DyRef}, any reflection subgroup $W'$ of $W$ (i.e. a subgroup $W'=\mpair{W'\cap T}$) has a canonical set of Coxeter generators \be \chi(W')=\mset{t\in T\mid N(t)\cap W'=\set{t}}\ee (with respect to the  simple reflections $S$ of $W$). Always consider $W'$ as Coxeter group with simple reflections  $\chi(W')$, unless otherwise stated. Recall that  $W'$ has a root system
$\Phi_{W'}:=\mset{\a\in \Phi\mid s_{\a}\in W}$ (in the class of  root systems considered in \cite{Sd}) with positive roots $\Phi_{+}\cap \Phi_{W'}$ and simple roots $\Pi_{W'}:=\mset{\a\in \Phi_{+}\mid s_{\a}\in \chi(W')}$.
The maximal dihedral reflection subgroups  of $W$ are the dihedral  reflection subgroups
(i.e. those generated by two distinct reflections) which are maximal under inclusion amongst the dihedral reflection subgroups; equivalently, they are the reflection subgroups $W'$ 
with $\vert \Pi_{W'}\vert =2$ and $ \Phi_{W'}=\Phi\cap \bbR \Pi_{W'}$
(see \cite[Remark 3.2]{DyBru}). Let $\CM$ be the set of all maximal dihedral reflection subgroups. For $\a\in \Phi$, let $\CM_{\alpha}:=\mset {W'\in \CM\mid s_{\a}\in W'}$.

\ssect{Reflection subgroups and the Bruhat graph}\lb{ss5.4} The following Lemma collects some  basic facts about cosets of   reflection subgroups in relation to the Bruhat graph. \begin{lem*} \begin{num}
\item  Let $W'$ be a reflection subgroup of $W$, $S':=\chi(W')$. For any $x\in W$, there is a unique element $u=x'_{W'}$ of $W'x$ of minimal length $l(u)$.
For any $w\in W'$, one has $N_{(W,S)}(wu)\cap W'=N_{(W',S')}(w)$. 
The map $z\mapsto zu$ induces an isomorphism of edge-labelled directed graphs
$\Omega_{(W',S')}\xrightarrow{\cong}\Omega_{(W,S)}(W'x)$. \item Let $(x,y)\in E$ and $\a=L(x,y)$.
For any $W'\in \CM_{\a}$, let $u=x'_{W'}$ denote the element of minimal length in $W'x=W'y$,
$x_{W'}:=xu^{-1}\in W'$ and   $y_{W}:=yu^{-1}=s_{\a}x_{W'}\in W'$.
Then $x_{W'}\leq_{W',\eset}y_{W'}$ and the map $z\mapsto zu$ induces an isomorphism of edge-labelled directed graphs 
$\Omega_{W'}([x_{W'},y_{W'}]_{W',\eset})\rightarrow \Omega_{W}([x,y]_{W,\eset}\cap xW')$.
\item If $(x,y)\in E$, $L(x,y)=\a$  and $l(y)-l(x)\geq 3$, there is some $W'\in \CM$  such that
$l_{W'}(y_{W'})-l_{W'}(x_{W'})\geq 3$ where $l_{W'}$ is the length function on $(W',\chi(W'))$.
\end{num}
\end{lem*}
\begin{proof} Part (a) is  from \cite{DyRef} and \cite{DyBru}, and (b) follows from (a)
(cf. also \cite[(1.4)]{DyHS2}). An ad hoc proof of (c) is given in \cite{DyBru}. A more natural  argument for (c) is to note that in the  following  identity, which holds  for $(x,y)\in E$ with  $y=s_{\a}x$,  $\a\in \Phi_{+}$,
\be \lb{eq5.4.1} l(y)-l(x)-1=\sum_{W'\in \CM_{\a}}(l_{W'}(y_{W'})-l_{W'}(x_{W'})-1)\ee  the left hand side and terms on the right are even elements of $\bbN$
since  $y_{W'}=s_{\a}x_{W'}>_{W',\eset}x_{W'}$.  This identity and  argument  are given in a  more general context as \cite[(1.2.1) and  (2.7)--(2.8)]{DyHS2}. 
A simple direct proof of \eqref{eq5.4.1} in the special situation here can be given as follows. Since $\alpha\in \Phi_{+}$,  it follows that $\Phi_{+}\sm\set{\a}=\dot \cup_{W'\in \CM_{\a}}(\Phi_{W',+}\sm\set{\a} )$.
Since  $\a\in \Phi_{y}$, (a) implies that 
 \begin{multline*} l(y)-1=\vert \Phi_{y}\sm \set{\a}\vert =\sum_{W'\in \CM_{\a}}\vert (\Phi_{y}\sm \set{\a})\cap \Phi_{W'<+}\vert\\ =\sum_{W'\in \CM_{\a}}\vert \Phi_{y_{W'},{W'}}\sm \set{\a}\vert=
\sum_{W'\in \CM_{\a}}(l_{W'}(y_{W'})-1)\end{multline*} Similarly, since $\a\not\in \Phi_{x}$
 \bee l(x)=\vert \Phi_{x}\vert =\sum_{W'\in \CM_{\a}}\vert \Phi_{x}\cap \Phi_{W',+}\vert\\ =\sum_{W'\in \CM_{\a}}\vert \Phi_{x_{W'},{W'}}\vert =
\sum_{W'\in \CM_{\a}}l_{W'}(x_{W'})\eee and \eqref{eq5.4.1} follows on subtracting.
\end{proof}

\subsection{Closure and the Bruhat graph (proof)}
\begin{proof}[of Lemma \ref{ss1.7}]   For  any $\G\seq\Phi_{+}$, write 
${\widehat{\G}}=\cup_{\a,\b\in \G}\ol{\set{\a,\b}}$. Recall that $\Phi_{x,n}=\eset$ for even $n$. By Lemma \ref{ss3.2}(a),  it will suffice to prove the following two assertions for any $x\in W$ and $n\in \bbN$:

\begin{num} \item Let $\G\seq\Phi_{x}'$.
Then ${\widehat{\G}}\cap \Phi_{x,1}=\G \cap \Phi_{x,1}$, and  if $\G\sreq \cup_{j\in \bbN_{\leq n}}\Phi_{x,2j+1}$,
then ${\widehat{\G}}\sreq \cup_{j\in \bbN_{\leq n+1}}\Phi_{x,2j+1}$
 \item Let $\G\seq\Phi_{x}$.
Then ${\widehat{\G}}\cap \Phi_{x,-1}=\G \cap \Phi_{x,-1}$, and   if $\G\sreq \cup_{j\in \bbN_{\leq n}}\Phi_{x,-(2j+1)}$,
then ${\widehat{\G}}\sreq \cup_{j\in \bbN_{\leq n+1}}\Phi_{x,-(2j+1)}$\end{num}
The proof of both parts is by reduction to the case of dihedral groups using Lemma \ref{ss5.4}; the proof  is given only for (a), since that of (b) is entirely similar.

Let $\G\seq \Phi'_{x}$. Obviously,  $\G\seq {\widehat{\G}}$ implies that  $\G\cap \Phi_{x,1}\seq{\widehat{\G}} \cap \Phi_{x,1}$.
For the reverse inclusion, suppose $\a\in ({\widehat{\G}}\sm\G) \cap \Phi_{x,1}$.
Then $\a=c\b+d\g$ where $c,d\in \bbR _{\geq 0}$  and $\b,\g\in \G$.
Since $\a\not \in \G$, it follows that $c,d>0$ and $\b \neq \g$.
There is $W'\in \CM$ with $\Phi_{W'}=\Phi\cap(\bbR  \b+\bbR \g)$. In fact, $W'\in \CM_{\a}$. 
 For $\d\in \set{\a,\b,\g}$, we have  $l(s_{\d}x)>l(x)$, and so  from Lemma \ref{ss5.4}(b),  $l_{W'}(s_{\d}xu)>l_{W'}(xu)$ where $u:=x_{W'}^{\prime-1}$. Suppose notation is  chosen so that
 $l_{W'}(s_{\b}xu)\leq l_{W'}(s_{\g}xu)$.
From the well-known descriptions of  dihedral groups and their root systems, one checks that the
above conditions imply that $l_{W'}(s_{\a}xu)> l_{W'}(s_{\b}xu)$
and hence there is a path of non-zero  length  from $s_{\b}xu$ to $s_{\a}xu$ in $\Omega_{W'}$.
From \ref{ss5.4}(b) again, it follows that there is a path of non-zero length from $s_{\b}x$ to $s_{\a}x$ in $\Omega$
and so $l(x)+1=l(s_{\a}x)> l(s_{\b}x)\geq l(x)+1$, a contradiction which completes the proof that
$\G\cap \Phi_{x,1}={\widehat{\G}} \cap \Phi_{x,1}$.

To prove the second part of (a), take $\G\seq \Phi'_{x}$ with $\G\sreq \cup_{j\in \bbN_{\leq n}}\Phi_{x,2j+1}$. Let $\a\in \Phi_{x,2n+3}$ i.e. $\alpha\in \Phi_{+}$ with $l(s_{\a}x)=l(x)+(2n+3)$. It will suffice to show that $\a\in {\widehat{\G}}$.
Since  $l(s_{\a}x)-l(x)\geq 3$,  Lemma \ref{ss5.4}(c) implies that there exists $W'\in \CM_{\a}$ such that  
$l_{W'}(s_{\a}xu)-l_{W'}(xu)\geq 3$ where $u:=x_{W'}^{\prime-1}$. Now there are  distinct  roots $\beta,\gamma\in \Phi_{W',+}$ such that $l_{W'}(s_{\b}xu)=l_{W'}(s_{\g}xu)=l_{W'}(xu)+1$.
Again using  the descriptions of dihedral groups and their  root systems, one checks that $\a\in \bbR _{>0}\b+\bbR _{>0}\g$ and that there are paths of non-zero lengths in $\Omega_{W'}$ from  $s_{\b}xu$ to $s_{\a}xu$  and from 
 $s_{\g}xu$ to $s_{\a}xu$.
 By  Lemma \ref{ss5.4}(b) again, there are paths in $\Omega$ of non-zero length  from
 $s_{\b}x$ to $s_{\a}x$ and from $s_{\g} x$ to $s_{\a}x$.
 Hence $l(x)<l(s_{\b} x)<l(s_{\a}x)=l(x)+2n+3$ and 
 $l(x)<l(s_{\g} x)<l(s_{\a}x)=l(x)+2n+3$. This implies that $\b,\g\in  \cup_{j\in \bbN_{\leq n}}\Phi_{x,2j+1}\seq \G$ and so $\a\in \ol{\set{\b,\g}}\seq {\widehat{\G}}$ as claimed.\end{proof}

 \section{Closure and meets} \lb{S6}
 In this section, Theorem \ref{ss1.5}(b) is deduced from the more general statement Theorem \ref{ss6.1}, which is itself a special case of Conjecture \ref{ss1.14}.
   \ssect{Cofinite closures of coclosed sets are biclosed}
 \lb{ss6.1}
 \begin{thm*} Let $\G$ be any coclosed subset of $\Phi_{+}$ such that $\ol{\G}$ has finite complement in $\Phi_{+}$. Then $\ol{\G}$ is biclosed i.e $\ol{\G}=\Phi_{x}'$ for some
 $x\in W$.\end{thm*}
 \begin{proof} The following trivial fact is used below. If $v$ is a non-minimum, non-maximum element of a dihedral reflection subgroup $W'$ in its  Chevalley-Bruhat order $\leq _{W',\eset}$, then one may
write $\Pi_{W'}=\set{\a',\b'}$ where $s_{\a'}v<_{W',\eset}v<_{W',\eset}s_{\b'}v$.

 The Theorem is proved  by induction on $n:=\vert \Phi_{+}\setminus \ol{\G}\vert$.
  If $n=0$, then $\ol{\G}=\Phi_{+}$ is obviously biclosed.
 Suppose next that $n>0$ i.e $\ol{\G}\neq \Phi_{+}$. By Lemma \ref{ss1.7}(b), there is $\a\in \Pi\sm\G$.
 Abbreviate $s:=s_{\a}$.
 We claim  that  $\G'=s \cdot \G$ is coclosed; this follows by a similar reduction to the dihedral case  as that for Lemma \ref{ss3.1}(c) (note $\a\not\in \G$ is essential this time, since it is so in the dihedral case). 
Since $\a\not\in \G\seq \Phi_{+}$, it follows that $\a\not \in \ol{\G}$.  Hence \bee\ol{\G'}=\ol{\set{\a}\cup s (\G)}\sreq \set{\a}\cup\ol{ s (\G)}=\set{\a}\cup s (\ol{\G}).\eee This shows that the map
$\beta\mapsto s (\beta)$ indices an injection $\Phi_{+}\setminus \ol{\G'}\rightarrow
\Phi_{+}\setminus(\ol{\G}\dotcup\set{\a})$ and hence that $\vert \Phi_{+}\sm \ol{\G'}\vert <\vert \Phi_{+}\sm\ol{ \G}\vert$. By induction, there exists $x\in W$ such that $\ol{\G'}=\Phi'_{s x}$. Since $\a\in \ol{\G'}$, this implies that $\a\in \Phi'_{s x}$, $\a\not \in \Phi_{s x}$ and $\a\in \Phi_{x}$ i.e. $l(s x)<l(x) $. To prove the theorem, it will be shown that $\ol{\G}=\Phi_{x}'$.

By Lemma \ref{ss1.7}(b),  ${\G'}\sreq
\Phi_{s x,1}$ and it will suffice   to show that ${\G}\sreq \Phi_{x,1}$. Let $\b\in \Phi_{x,1}$ i.e. $\beta\in \Phi_{+}$ with $l(s_{\b}x)=l(x)+1$. Let $z:=s_{\b}x$.  By the Z-property of Bruhat order (see \cite{DyHS2}, for example) and the fact each length two Bruhat interval has $4$ elements, the situation is as in one of the following two diagrams, which show   vertices,  edges and edge labels appearing in paths (all paths from $sx$ to $z$ if $l(sz)<l(z)$, or from $sx$ to $sz$ if $l(sz)>l(z)$) in   the Bruhat graph $\Omega$.
\bee 
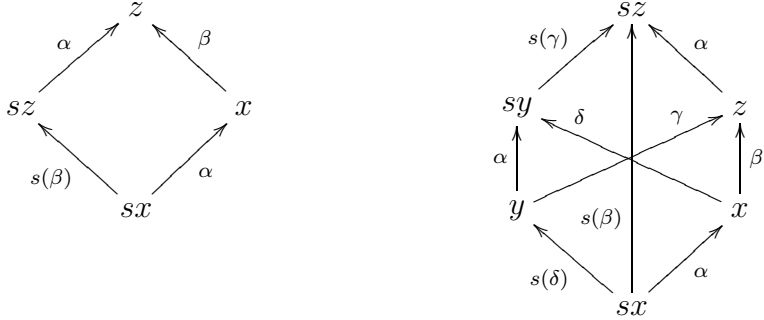
\xymatrix{ &{z}&&&& &{sz}&&\\
{sz}\ar@{->}[ur]^{\a}&&
{x}\ar@{->}[ul]_{\b}&&&
{sy}\ar@{->}[ur]^{s(\g)}&&
{z}\ar@{->}[ul]_{\a}\\
&{sx}\ar@{->}[ul]^{s(\b)}\ar@{->}[ur]_{\a}&&&&
{y}\ar@{->}[u]^{\a}\ar@{->}[urr]^(.75){\g}&&
{x}\ar@{->}[u]_{\b}\ar@{->}[ull]_(.75){\d}\\
&&&&&&
{sx}\ar@{->}[uuu]^(.3){s(\b)}\ar@{->}[ul]^{s(\d)}\ar@{->}[ur]_{\a}&} \eee

Consider first the case that $l(sz)<l(z)$. Then $s(\beta)\in \Phi_{sx,1}\sm\set{\a}\seq\G'\sm\set{\a}=s(\G)$, so $\b\in \G$ as desired in this case.
Consider now the contrary case that $l(sz)>l(z)$. Let $W'\in \CM$ with $\Phi_{W'}=\Phi\cap(\bbR  \a+\bbR \b)$. 
 Multiplying the vertex labels of this diagram by $u:=x_{W'}^{\prime-1}$ on the right gives a corresponding
diagram in $\Omega_{W'}$. One necessarily has $l_{W'}(szu)=l_{W'}(sxu)+3$. 

Since $\a\in \Pi_{W'}$, inspecting the vertex $xu$ of the resulting diagram and using the trivial fact at the start of the proof  shows that either $\Pi_{W'}=\set{\a,\d}$ or $\Pi_{W'}=\set{\a,\b}$.
The first case can't occur since then $syu$ would be the longest element of $W'$, which it isn't since it is the initial vertex of an edge $(syu,szu)$.  Therefore,  $\Pi_{W'}=\mset{\a,\b}$ and $\d\not \in \Pi_{W'}$.  Now $\G\cap \Phi_{W',+}$ is coclosed in $\Phi_{W',+}$, and 
 from the first case, it follows that $\d\in \G\cap \Phi_{W',+}$. By examining the possible coclosed sets  in dihedral groups, one sees that this  implies that either 
  $\a\in \G\cap \Phi_{W',+}$ (which is false here since $\a\not\in \G$) or $\b\in \G\cap \Phi_{W',+}$.
  Hence $\b\in \G$ whether $l(sz)<l(z)$ or $l(sz)>l(z)$. This shows that $\Phi_{x,1}\seq \G$ and completes the proof.\end{proof}

 \subsection{Closure and meets (proof)}
 \begin{proof}[of Theorem \ref{ss1.5}(b)] To prove Theorem \ref{ss1.5}(b), it will suffice to show that $\ol{\D}$ is biclosed, where 
 Let $X$ be a non-empty subset of $W$ and 
 $\D:=\cup_{x\in X}\Phi'_{x}$. Note that $\ol{\D}$  has finite complement 
 in $\Phi_{+}$ (as $\D$ itself does, since $\vert X\vert \geq 1$).
  Also,  $\cup_{x\in X}\Phi'_{x}$ is a union of biclosed sets, so it is coclosed.
 and
  therefore it  is of the form $\ol{\D}=\Phi'_{y}$  for some $y\in W$
 by Theorem \ref{ss6.1}.  One may check that $y=\meet X$ in weak order as follows.
 For any $x\in X$, $\Phi'_{x}\seq \ol{\D}=\Phi_{y}'$ implies $\Phi_{y}\subseteq \Phi_{x}$, so
 $y$ is a lower bound for $X$. On the other hand, if $z$ is any lower bound of $X$  then $\Phi_{z}\seq \Phi_{x}$ for all $x\in X$,
 so $\D\subseteq \Phi'_{z}$. Taking closures, $\Phi'_{y}=\ol{\D}\seq \Phi'_{z}$
 so $\Phi_{z}\seq \Phi_{y}$ and $z\leq y$. Hence $y=\meet X$ and \ref{ss1.5}(b)  follows.
 \end{proof}
 \ssect{Closure of the union of a finite biclosed set  and a cofinite biclosed set} \label{ss6.2} The following Corollary is another special case of Conjecture  \ref{ss1.14}(a).
  \begin{cor*} Let $x,y\in W$. Then $\ol{\Phi_{x}\cup \Phi'_{y}}=\Phi_{z}'$ for some $z\in W$. The element $z$ is the maximum element in weak order of the set $\mset{w\in W\mid w\leq y, l(x^{-1}w)=l(x)+l(w)}$.
  \end{cor*}\begin{proof} This follows easily  from the theorem by  taking $\G:=\Phi_{x}\cup \Phi_{y}'$. The details are omitted.\end{proof}
  \section{Closure of the  union of a biclosed set and  a root} \lb{S7}
In general, if $\G$ is biclosed and $\a\in \Phi_{+}$, then $\ol{\G\cup\set{\a}}$ need not be  biclosed. There need not even be   a unique inclusion-minimal biclosed set containing $\ol{\G\cup\set{\a}}$; for example, consider $\G=\emptyset$, $(W,S)$ of type $A_{2}$  and $\a$ the highest root. Theorem \ref{ss1.8}, which is proved in this section,  gives sufficient (but far from necessary) conditions to ensure that such a closure is biclosed when either  $\G$ or $\Phi_{+}\sm \G$  finite.
\ssect{Closures in dihedral groups} \label{ss:trivdihed} The proof of Theorem \ref{ss1.8} is by reduction to the case of dihedral groups.
The following trivial  Lemma isolates some relevant  properties of dihedral groups for the proof of  \ref{ss1.8}(a).\begin{lem*} Assume that $(W,S)$ is dihedral. Let $\a\in \Phi_{+}$\begin{num} 
\item  For $w\in W$ with $l(s_{\a}w)=l(w)+1$, 
exactly one of the following  three possibilities occurs:\begin{subconds} 
\item $\a\in \Pi$, $w=1_{W}$ and $\ol{(\Phi_{w}\cup\Phi_{s_{a}w})\sm\set{\a}}=\eset$
\item $\a \in \Pi$, $w\neq 1_{W}$ and $\ol{(\Phi_{w}\cup\Phi_{s_{a}w})\sm\set{\a}}=\Phi_{+}\sm\set{\a}$
\item $\a\not \in \Pi$, $\Phi_{s_{a}w}=\Phi_{w}\dotcup\set{\a}$ and  $\ol{(\Phi_{w}\cup\Phi_{s_{a}w})\sm\set{\a}}=\Phi_{w}$.
\end{subconds} 
\item Assume that $\a\not\in \Pi$. If $W$ is infinite, there is a unique $w\in W$ satisfying $\text{\rm (a)(iii)}$. If $W$ is finite, there are exactly  two elements $w\in W$ satisfying $\text{\rm(a)(iii)}$; denoting them as $w'$ and $w''$, one has $\Phi_{+}=\Phi_{w'}\dotcup\set{\a}\dotcup \Phi_{w''}$. \end{num} 
\end{lem*}
 
  \ssect{Closure after adjoining a root (proof)}
  \begin{proof}[of Theorem \ref{ss1.8}]\lb{ss7.2}
We prove (a).  As in Lemma \ref{ss5.4},  for any any $W'\in \CM_{\a}$ and $p\in W$, write
$p=p_{W'}p'_{W'}$ where $p_{W'}\in W'$ and  $p'_{W'}$ is the element of minimal length in $W'p$.  Also write
\bee \Phi_{p,W'}:=\Phi_{p}\cap \Phi_{W'}=\Phi_{W',+}\cap p'_{W'}(\Phi_{W',+})\eee  where the right hand  equality is by  Lemma \ref{ss5.4}(a).
Note $\a\not\in \Phi_{x}$. 
Set $z=s_{\a}x$. Then  $z'_{W'}=x'_{W'}$ and  $z_{W'}=s_{\a}x_{W'}$.
From \ref{ss5.4}(b), it follows  that $l_{W'}(s_{\a}x_{W'})=l_{W'}(z_{W'})=l_{W'}(x_{W'})+1$.
 Lemma \ref{ss:trivdihed} implies that $\ol{\Phi_{x_{W'},W'}\cup\set{\a}}=\ol{\Phi_{x_{W'},W'}\cup\Phi_{s_{\a}x_{W'},W'}}$.  So
\bee \ol{\Phi_{x}\cup\set{\a}}=\ol{\cup_{W'\in \CM_{\a }}\Phi_{x_{W'},W'}\cup
\set{\a}}= \ol{\cup_{W'\in \CM_{\a}}(\Phi_{x_{W'},W'}\cup \Phi_{s_{a}x_{W'},W'})}=\ol{\Phi_{x}\cup\Phi_{s_{\a}x}}.\eee Since $\Phi_{x}\cup\set{\a}\seq \Phi_{v}$, it follows that  
$\ol{\Phi_{x}\cup\Phi_{s_{\a}x}}=\ol{\Phi_{x}\cup\set{\a}}\seq \Phi_{v}$ i.e. $v$ is an upper bound for $x$ and $s_{\a}x$ in weak order. Let $y:=x\vee s_{\a}x$, so  by Theorem \ref{ss1.5}(a),
$\ol{\Phi_{x}\cup\Phi_{s_{\a}x}}=\ol{\Phi_{x}\cup\set{\a}}=\Phi_{y}\seq \Phi_{v}$.
If $z\in W$ with $x\leq z$ and ${\a}\in \Phi_{z}$, then
$\Phi_{y}=\ol{\Phi_{x}\cup\set{\a}}\seq \Phi_{z}$. Hence
$y$ is the minimal element of $\mset{z\in W\mid x\leq z, {\a}\in \Phi_{z}}$.

Clearly, $\a\in \Phi_{y}$ so $l(s_{\a}y)<l(y)$. Since $x\leq y$, it is possible to write 
$y=xu$ where $l(y)=l(x)+l(u)$. Note that  $u\neq 1_{W}$ since $\a\in \Phi_{y}\sm \Phi_{x}$.
Choose $\tau\in \Pi$ such that $l(us_{\tau})<l(u)$. Then $x\leq ys_{\tau}<y$.
Since $\Phi_{y}=\ol{\Phi_{x}\cup\set{\a}}$, it follows that $\a\not \in \Phi_{ys_{\tau}}$ and therefore
$\a\in \Phi_{y}\sm \Phi_{ys_{\tau}}=\set{-y(\tau)}$. In particular,
$s_{\a}y=ys_{\tau}<y$ and  $\Phi_{s_{\a}y}=\Phi_{y}\sm\set{\a}$.
Obviously, $\ol{(\Phi_{x}\cup\Phi_{s_{\a}x})\sm\set{\a}}\seq  \Phi_{y}\sm\set{\a}=\Phi_{s_{\a}y}$ since  $\Phi_{y}\sm\set{\a}$ is closed.  To complete the proof of (a), it will suffice to verify  the following claim: 
\be \lb{eq7.2.1}\Phi_{s_{\a}y}\seq\ol{(\Phi_{x}\cup\Phi_{s_{\a}x})\sm\set{\a}}.\ee
From Lemma \ref{ss1.7}, it follows  both  that $\Phi_{x}\cup \set{\a}\sreq \Phi_{y,-1}$ and that it will suffice to show that $(\Phi_{x}\cup\Phi_{s_{\a}x})\sm\set{\a}\sreq \Phi_{s_{\a}y,-1}$.

Abbreviate $s_{\a}=s$, $s_{\tau}=r$ so $sy=yr$ and $r\in S$. Let $\b\in \Phi_{sy,-1}$ i.e. $\b\in \Phi_{+}$ with $l(s_{\b}sy)=l(sy)-1$.
Set $z:=s_{\b}sy$. Similarly as in \ref{ss6.1}, we have one of the possibilities indicated by the following diagrams of  vertices,  edges and  edge labels appearing in paths (all paths  from $z$ to $y$ if $l(zr)>l(z)$, or from $zr$ to $y$ if $l(zr)<l(z)$) in   the Bruhat graph $\Omega$.

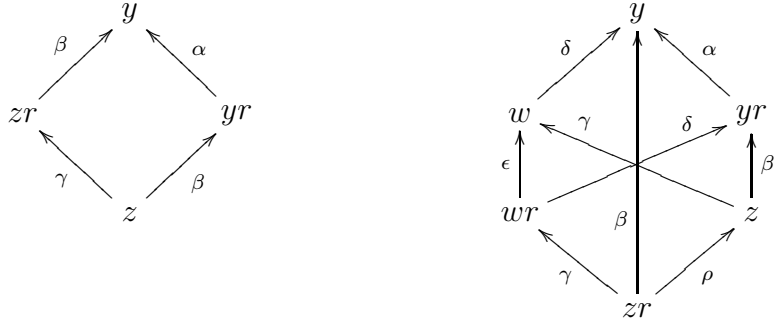
 \bee \xymatrix{ &{y}&&&& &{y}&&\\
{zr}\ar@{->}[ur]^{\b}&&
{yr}\ar@{->}[ul]_{\a}&&&
{w}\ar@{->}[ur]^{\d}&&
{yr}\ar@{->}[ul]_{\a}\\
&{z}\ar@{->}[ul]^{\g}\ar@{->}[ur]_{\b}&&&&
{wr}\ar@{->}[u]^{\e}\ar@{->}[urr]^(.75){\d}&&
{z}\ar@{->}[u]_{\b}\ar@{->}[ull]_(.75){\g}\\
&&&&&&
{zr}\ar@{->}[uuu]^(.3){\b}\ar@{->}[ul]^{\g}\ar@{->}[ur]_{\rho}&} \eee

Consider first the case that $l(zr)>l(z)$. Then $\beta\in \Phi_{y,-1}\sm\set{\a}\seq\Phi_{x}$, so $\b\in (\Phi_{x}\cup \Phi_{sx})\sm\set{\a}$ as desired in this case.
Consider now the contrary case that $l(zr)<l(z)$. Let $W'\in \CM$ with $\Phi_{W'}=\Phi\cap(\bbR  \a+\bbR \b)$. Multiplying the vertex labels of the second diagram by $u:=y_{W'}^{\prime-1}$ on the right gives a corresponding
diagram in $\Omega_{W'}$. One necessarily has $l_{W'}(yu)=l_{W'}(zru)+3$.
Note that $s_{\a}(x_{W'})=(s_{a}x)_{W'}$ and $s_{\a}(y_{W'})=(s_{\a}y)_{W'}$. 
By the argument  in the first case, $\d\in \Phi_{x}$.  
Hence $\d\in \Phi_{x,W'}$ and $x_{W'}\neq 1_{W'}$. If $\a\in \Pi_{W'}$, then Lemma \ref{ss:trivdihed}(a) gives
we have \bee \b\in \Phi_{W',+}\sm\set{\a}=\ol{(\Phi_{x,{W'}}\cup\Phi_{s_{a}x,{W'}})\sm\set{\a}}\seq \ol{(\Phi_{x}\cup\Phi_{s_{\a}x})\sm\set{\a}}\eee as required. So we may, and do,  assume that $\a\not \in \Pi_{W'}$. 
 Lemma 
\ref{ss:trivdihed}(a) implies that $\d\in \Phi_{x,W'}=\Phi_{s_{\a}x,W'}\sm\set{\a}$. Also,  since $\Phi_{s_{\a}y}=\Phi_{y}\sm\set{\a}$, it follows that  $\d\in \Phi_{s_{\a}y,W'}=\Phi_{y,W'}\sm\set{\a}$. Thus,  Lemma 
\ref{ss:trivdihed}(b) implies that $ x_{W'}=s_{\a}y_{W'}$. From the above diagram,
$\b\in \Phi_{y}$ so \bee\b\in \Phi_{y,W'}\sm\set{\a}=\Phi_{s_{\a}y,{W'}}=\Phi_{x,{W'}}\seq\ol{(\Phi_{x,{W'}}\cup\Phi_{s_{\a}x,{W'}})\sm\set{\a}}\seq \ol{(\Phi_{x}\cup\Phi_{s_{\a}x})\sm\set{\a}}\eee  as required.

The proof of Theorem \ref{ss1.8}(b) is very similar to that of \ref{ss1.8}(a)   and is omitted.
\end{proof}
\ssect{Fibering the closure  over a newly adjoined root} \label{ss7.3}
Theorem \ref{ss1.8} implies that $\Phi_{y}$ in \ref{ss1.8}(a) (resp., $\Phi'_{y}$ in \ref{ss1.8}(b)) is well fibered over $\a$ in the following sense: for any  $W'\in \CM_{\a}$, the intersection  of  $\Phi_{y}$ (resp., $\Phi'_{y}$) with  the plane spanned by the roots of $W'$ consists of all positive roots lying in a fixed one  of the two  closed half-planes in that plane bounded by the line spanned by $\a$; for all but  the  finitely many $W'\in \CM_{\a}$
 for which $s_{\a}\not\in \Pi_{W'}$, this intersection is either $\Phi_{W',+}$  or $\set{\a}$.
 
  Higher dimensional analogues of this phenomenon, involving closures of 
  sets obtained by adjoining all roots of a suitable reflection subgroup,   may 
  be expected but remain conjectural in general. The Corollary below is the 
  simplest  result  of this type.
  
  \begin{cor*} Let $W_{J}$ be a finite parabolic subgroup of $W$ with longest element $w_{J}$. Let $x\in W$ such that $l(w_{J}x)=l(w_{J})+l(x)$ and 
  $y:= w_{J}\vee x$ exists in  weak order on $W$. Write $y=w_{J}z$. Then
 \begin{num}\item $u\vee z=uz$  and $\Phi_{u\vee z}=\Phi_{u}\dotcup \Phi_{z}$ for all $u\in W_{J}$.
 \item For any $\a\in \Phi_{W_{J},+}$ and any $W'\in \CM_{\a}$, either
 $\Phi_{W',+}\seq \Phi_{y}$ or $\Phi_{W',+}\cap \Phi_{y}=\set{\a}$.  
  \end{num} 
   \end{cor*} 
 \begin{proof} Note that $x$ and $z$ are the minimal length elements of their cosets $W_{J}x$ and $W_{J}z$ respectively, so for any $u\in W_{J}$,
 $l(ux)=l(u)+l(x)$ and $l(uz)=l(u)+l(z)$.
 
 It is well-known that the  map $s\mapsto w_{J}sw_{J}$ defines a bijection $\theta \colon J\to J$.  Now for any $s\in J$,
 \begin{equation*} w_{J}z= w_{J}\vee x=(s\vee w_{J\setminus{\set{s}}})\vee x=
 s\vee (w_{J\setminus{\set{s}}}\vee x).
 \end{equation*} Since $\Phi_{w_{J\setminus{\set{s}}}}\cap \Phi_{s}=\eset=\Phi_{x}\cap \Phi_{s}$, Corollary \ref{ss1.6} implies that 
$\Phi_{ w_{J\setminus{\set{s}}}\vee x}\cap \Phi_{s}=\eset$. From Theorem \ref{ss1.8}, it follows that $w_{J}z=s\vee sw_{J}z$. Using that $\theta$ is a bijection, 
\begin{equation*}
(w_{J}s)z=\theta(s)w_{J}z\leq \theta(s)\vee \theta(s)w_{J}z=  w_{J}z=(w_{J}s)(sz).
\end{equation*} Since $l((w_{J}s)z)=l(w_{J}s)+l(z)$ and 
$l\bigl((w_{J} s))(sz)\bigr)=l(w_{J}s)+l(sz)$, this implies that
$z\leq s z$. 
Write $sz=zs'$ where $s'\in S$.   Varying $s$ gives a bijection
$s\mapsto s'\colon J\to K$ for some subset $K$ of $S$.  This bijection extends to a 
group isomorphism $W_{J}\to W_{K}$ which will  be denoted as $u\mapsto u'$. Note that $uz=zu'$ for all $u\in W_{J}$. Now $z$ is the minimal length element in $W_{J}z$ and in $zW_{K}$. Let $u\in W_{J}$. Then  $l(uz)=l(zu')=l(z)+l(u)=l(z)+l(u')$.  This implies that $z\leq uz$, $u\leq uz$ and  therefore $ z\vee u\leq uz$. Since \begin{equation*}
l(z\vee u)=\vert \Phi_{z\vee u}\vert \geq\vert \Phi_{z}\cup \Phi_{u}\vert  =\vert \Phi_{z}\dotcup \Phi_{u}\vert =\vert \Phi_{z}\vert+\vert \Phi_{u}\vert  =l(u)+l(z)=l(uz),
\end{equation*} it follows that $z\vee u=uz$ and $\Phi_{u\vee z}=\Phi_{u}\dotcup \Phi_{z}$, proving (a).

Now let $\a,W'$ be as in (b). If $W'\seq W_{J}$, then $\Phi_{W',+}\seq \Phi_{W_{J},+}=\Phi_{w_{J},+}\seq \Phi_{y}$. Otherwise, $\Phi_{W_{J}}\cap \Phi_{W',+}=\set{\a}$ (since $W_{J}$ finite parabolic implies $\Phi_{W_{J}}=\Phi\cap \mathbb{R} \Phi_{W_{J}}$). This implies that $\a\in \Pi_{W'}$.  Choose $u\in W_{J}$, $r\in J$ so
that $s_{\a}=uru^{-1}$ and $l(s_{\a})=2l(u)+1$. Then  $s_{\a}u=ur>u$, so $\overline{\set{\a}\cup\Phi_{u}}=\Phi_{s_{\a}u}$. Also, $s_{\a}uz=urz=uzr'>uz$
so $\overline{\set{\a}\cup \Phi_{uz}}=\Phi_{s_{a}uz}$.  From the remarks at the start of this subsection, either $\Phi_{W'}\cap \Phi_{s_{\a}uz}=\Phi_{W',+}$ or $\Phi_{W'}\cap \Phi_{s_{\a}uz}=\set{\a}$. Since $ \Phi_{s_{a}uz}=\Phi_{s_{\a u}}\dotcup \Phi_{z}$ and  $ \Phi_{y}=\Phi_{w_{J}z}=\Phi_{w_{J}}\dotcup \Phi_{z}$ with $\Phi_{s_{\a}u}\cap \Phi_{W',+}=\set{\a}=\Phi_{w_{J}}\cap \Phi_{W',+}$, (b) follows.

 \end{proof}

 \section{Galois connections} \lb{S8}
In this section, Theorem \ref{ss1.9} and its  dual are proved.
 \ssect{First  Galois connection: proof}
 \begin{proof}[of Theorem \ref{ss1.9}]
 Assume that $x,z\in W$ with   $xRz$ i.e. $z(\Phi_{x})=\Phi_{x}$.
 Then $\Phi_{x}\cap \Phi_{z^{-1}}=\eset$ and $l(zx)=l(z)+l(x)$. Also,  $xRz^{-1}$ holds since $z^{-1}(\Phi_{x})=\Phi_{x}$. This proves (b). 
 Note that 
 $\Phi_{x}\cup \Phi_{z}=\Phi_{z}\dotcup z(\Phi_{x})=\Phi_{zx}$
 so clearly $z\vee x=zx$. This proves the left implication  ``$\Longrightarrow$'' in (a).
 For the reverse  implication, suppose that $x\vee z=zx$ and  $\Phi_{x}\cap \Phi_{z}=\eset$. From the first equation, $z\leq zx$ so $l(zx)=l(z)+l(x)$.
Also, the two equations imply that  $\Phi_{x}\dotcup\Phi_{z}\seq \Phi_{zx}$ where the left hand side has cardinality
$l(x)+l(z)$  and the right hand side has cardinality $l(zx)=l(z)+l(x)$. This implies that 
$\Phi_{x}\dotcup\Phi_{z}= \Phi_{zx}=\Phi_{z}\dotcup z(\Phi_{x})$  and so $z(\Phi_{x})=\Phi_{x}$. It remains to prove the equivalence on the right  of (a). 
The right implication   ``$\Longleftarrow$'' is trivial and  the converse follows 
since if $\Phi_{x}\cap \Phi_{z}=\eset$, then \bee\vert \Phi_{x\vee z}\vert \geq 
\vert \ol{\Phi_{x}\dotcup \Phi_{z}}\vert \geq \vert \Phi_{x}\vert +\vert \Phi_{z}\vert=l(z)+l(x)\geq l(zx)=\vert \Phi_{zx}\vert.\eee   This  completes the proof of (a).

 For $X\seq W$, $X^{\dag}=\mset{z\in W\mid z(\Phi_{x})=\Phi_{x}\text{ for all $x\in X$}}$
  is clearly a subgroup of $W$, proving (c). 
 To prove (d), consider a subset $Z$ of $W$ and let  \bee \CL:=Z^{*}=\mset{x\in W\mid x\vee z=zx,
 \Phi_{x}\cap \Phi_{z}=\eset\text{ for all $z\in Z$}}.\eee
  Obviously $1_{W}\in \CL$. Suppose given a non-empty subset $A$ of $\CL$, say with $a'\in A$, and let $a'':=\meet A$ in $(W,\leq)$. Let $z\in Z$. Note that $a''\leq a'$ implies that   $\Phi_{z}\cap \Phi_{a''}\seq \Phi_{z}\cap \Phi_{a'}= \eset$ for all $z\in Z$. Further,   $l(za)=l(z)+l(a)$  for all $a\in A$; in particular, since $a''\leq a'\in A$, it follows that $l(za'')=l(z)+l(a'')$
  and   \be \lb{eq8.1.1} z\vee  a''=z\vee \Bigl(\,\meet  A\,\Bigr)\leq \meet_{a\in A}(z\vee a)= \meet_{a\in A}za=z\Bigl(\,\meet_{a\in A}a\,\Bigl)=za''\ee  
  using Corollary  \ref{ss1.6}(a). Since $\Phi_{z}\cap \Phi_{a''}=\eset$, we have \bee l(z\vee a'')=\vert \Phi_{z\join a''}\vert\geq\vert \Phi_{z}\dotcup \Phi_{a''}\vert =\vert \Phi_{z}\vert +\vert\Phi_{a''}\vert=l(z)+l(a'')=l(za'')\eee and  equality holds   throughout \eqref{eq8.1.1}. Hence $a''\in \CL$ and so   $a''$ is obviously the greatest lower bound of $A$ in $\CL$. This proves (d). It remains to prove (e). Maintain the notation in the proof of (d), but suppose now that $A$ has  join $b$ in $(W,\leq)$. From $\Phi_{z}\cap \Phi_{a}=\eset$ and $l(za)=l(z)+l(a)$ for all $a\in A$, it follows that $\Phi_{b}\cap \Phi_{z}=\eset$
  and $l(zb)=l(z)+l(b)$ by  Corollary \ref{ss1.6}(b).
Hence   
 \bee z\vee  b=z\vee \Bigl(\,\join  A\,\Bigr)=\join_{a\in A}(z\vee a)= \join_{a\in A}za=z\Bigl(\,\join_{a\in A}a\,\Bigl)=zb\eee  by Corollary \ref{ss1.6}(a). This completes the proof of (e) and of the Theorem.
 \end{proof}
 \subsection{Second Galois connection} The dual result to Theorem  \ref{ss1.9} will now be formulated and its proof sketched.
 Define a relation $R'$ on $W$ by $xR'z$ if and only if  $z(\Phi'_{x})=\Phi'_{x}$  for $x,z\in W$.  
 Define the  two maps $X\mapsto X^{\dag'}$ and
$Z\mapsto Z^{*'}$ from $\CP(W)\rightarrow\CP(W)$ by replacing $R$ by $R'$ in the analogous definition in \ref{ss1.9}. Also define the corresponding families of stable subsets
\[\CW'_{*}:=\mset{\CL\in \CP(W)\mid \CL^{\dag' *'}=\CL}=\mset{Z^{*'}\mid Z\in \CP(W)}\]
and \[\CW'_{\dag}:=\mset{\CG\in \CP(W)\mid \CG^{*'\dag' }=\CG}=\mset{X^{\dag'}\mid X\in \CP(W)}.\] 
 \begin{thm*} \begin{num}\item One has \bee  xR'z\iff (\,\ol {\Phi'_{x}\cup\Phi_{z}}=\Phi'_{zx} \text{ and $\Phi'_{x}\cap \Phi_{z}=\eset)$}\iff  \Phi_{z}\dotcup \Phi'_{x}=\Phi'_{zx}.\eee
\item If $xR'z$ then $z\leq x$, $l(x)-l(z)=l(z^{-1}x)$ and $xR'z^{-1}$.  
\item The elements of $\CW'_{\dag}$ (other than perhaps $W$) are finite  subgroups of $W$.
\item The elements of $\CW'_{*}$ are (possibly empty) complete meet subsemilattices of $(W,\leq)$. One has $W^{*'}=\eset$ if $W$ is infinite, and otherwise $W^{*'}=\set{w_{S}}$ where $w_{S}$ is the longest element of $W$.
\item If $L\in \CW'_{*}$, then  for any  subset $X$ of $L$
which has an upper bound in $W$, its join $x=\join X$ in $W$ is an element of $L$ (and so $x$ is the least upper bound of $X$ in $L$).
\item Let $w\in\CI=\CI(W):=\set{w\in W\mid w^{2}=1_{W}}$. Then $wR'w$.  The stable subgroup
$\set{w}^{*'\dag'}$ is a  subgroup of $W$ contained in  $\mset{x\in W\mid x\leq w}$ and  containing $w$. The corresponding stable subsemilattice 
$\set{w}^{*'}$ is  contained in  $\mset{x\in W\mid x\geq w}$ and  has $w$ as minimum element. The map $w\rightarrow \set{w}^{*'}$ gives an injection $i\colon \CI(W)\rightarrow \CW'_{*}\setminus \set{\eset}$.   \end{num}\end{thm*}
\begin{proof} The proofs of (a)--(e) are similar to those of the corresponding parts of Theorem \ref{ss1.9} and most of  the details are omitted, except to remark that finiteness of the stable subgroups in (c) and the statement about $W^{*'}$ in (d) follow using that  $xR'z$  implies $z\leq x$. We  prove (f). Let $w\in W$. Then  \bee w^{-1}(\Phi'_{w})=w^{-1}(\Phi_{+}\cap w(\Phi_{+}))=\Phi_{+}\cap w^{-1}(\Phi_{+})=\Phi'_{w^{-1}}.\eee
Hence if $w^{2}=1$, then   $w(\Phi'_{w})=\Phi'_{w}$ i.e. $wR'w$.
This implies that $w$ is the minimum element of $\set{w}^{*'}$ and the maximum element of $\set{w}^{\dag'}$ in weak order. All statements of (f) follow readily.
\end{proof} 

\begin{rem*} (1) The map $i$ in (f) is not a bijection in general.  For  let $W$ be  of type $A_{4}$ with simple reflections $S=\set{r,s,t,u}$ and  Coxeter graph
\bee \xymatrix{r\ar@{-}[r]&s\ar@{-}[r]&t\ar@{-}[r]&u}.\eee 
One checks that $\mpair{s}^{\dag}=\mpair{u,rstsr}$, $\mpair{t}^{\dag}=\mpair{r,stuts}$ and
$\mpair{s,t}^{\dag}=\mpair{rstutsr}$ (by direct calculation, or by using results of \cite{HowBr}) where $\mpair{A}$ is the subgroup generated by $A$. This gives three stable subgroups in which   $rstutsr$ is an element of maximal length. But  if $i$ is  a bijection,  these stable subgroups would be uniquely  determined by their maximum element in weak order (which would be an involution), by (f).  

 (2) There is a partial order $\preceq$ defined on the set $\CI$  by $v\preceq w$ if and only if  $\set{v}^{*'\dag'}\seq\set{w}^{*'\dag'}$. Clearly,
$v\preceq w$ implies $v\leq w$, but  the reverse implication fails by \ref{ss2.2}.

(3) For finite $W$, the stable subgroups $X^{\dag'}$ or $X^{\dag} $ are  not necessarily subsemilattices of $W$ in weak order; they also need not be reflection subgroups of $W$.  Also, the stable subsemilattices $Z^{*}$ or $Z^{{*\prime}}$ need not be subgroups.

(4) The stable subgroups $X^{\dag}$ 
  need  not  be Coxeter groups for infinite $W$, as the groups $\set{s}^{\dag}$ for $s\in S$ may include non-trivial free groups by \cite{Br}.  We  do not know (for infinite or finite $W$)  if 
the (necessarily finite) stable subgroups $X^{\dag'}$ for $X\neq \eset$ are always  Coxeter   groups. \end{rem*}

\section{Rank one parabolic weak orders}\lb{S9}
This section gives a proof of Theorem \ref{ss1.11}.
\subsection{Joins with a simple reflection}  \lb{ss9.1}
The following Lemma collects   special cases and consequences of the main  results already proved, for use in the proof of Theorem \ref{ss1.11}
\begin{lem*} Let $s\in S$, say $s=s_{\a}$ where $\a\in \Pi$.
\begin{num} \item For $x\in W$, we have $x\in \set{s}^{*}$ if and only if  $s(\Phi_{x})=\Phi_{x}$ if and only if  
   $\Phi_{sx}=\Phi_{x}\dotcup\set{\a}$. 
\item  $\set{s}^{*}=\mset{x\in W\mid sx>x}$. 
  \item  $\set{s}^{*}$ is closed under taking meets, and those joins that exist, in $W$.
   \item If $x\in W$, $l(sx)=l(x)+1$ and $\set{s,x}$ has an upper bound, then $s\vee x=sy$ for some $y\in \set{s}^{*}$.\end{num}\end{lem*}
   \begin{proof} Part (a) follows from Theorem \ref{ss1.9}(a).
   For (b), note first that if $x\in \set{s}^{*}$, then $\Phi_{x}\seq \Phi_{sx}$ by (a)
   so $x< sx$. On the other hand, if $sx>x$, then $\a\not\in\Phi_{x}$
   so $\Phi_{sx}\supseteq \Phi_{x}\dotcup\set{\a}$. But $\vert \Phi_{sx}\vert=l(sx)=l(x)+1=\vert \Phi_{x}\dotcup\set{\a}\vert$ so equality holds, proving 
 (b). Part (c) is a special case of Theorem \ref{ss1.9}(d)--(e).
Part (d) follows from Theorem \ref{ss1.8}(a), Remark \ref{ss1.8} and (b).
\end{proof}

\subsection{Notation for a rank one parabolic weak order}\lb{ss9.2} For the remainder of this section, fix $s\in S$, say $s=s_{\a}$ where $\a\in \Pi$, and set $J=\set{s}$.
Let  $\L:=\L_{J}=\Phi_{+}\cup\set{-\a}=\Phi_{+}\cup s(\Phi_{+})$ and $\CL=\CL_{s}:=\CL_{\set{s}}$. Then $\Phi(\CL)=\set{\a,-\a}$ and $W(\CL)=\set{1,s}$.
In particular, every subset of $\Phi(\CL)$ is biclosed in both $\L$ and  $\Phi(\CL)$, the biclosed subsets of  $\Phi(\CL)$ form a complete lattice and 
$\t (\ol \G)=\t(\G)=\ol{\tau(\G)}$ is biclosed in $\Phi({\CL})$ for any $\G\seq \L_{J}$.
Recall that  $W(\CL)$ acts on $\CL$ as a group of order automorphisms by $(w,\L)\mapsto w(\L)$
satisfying $\t (w(\G))=w(\t (\G))$ for biclosed  subsets $\G$ of $\L$ and $w\in W(\CL)$. 

\ssect{Description of elements of the weak order by type}\lb{ss9.3}   The proof of Theorem \ref{ss1.11} uses the following Lemma   describing the elements $\G$ of $\CL_{s}$ according to their type $\t(\G)\seq\set{\a,-\a}$.
\begin{lem*} \begin{num}\item $\mset{\G\in \CL_{s}\mid \t(\G)=\eset}=\mset{\Phi_{x}\mid x\in W, x<sx}$ 
\item $\mset{\G\in \CL_{s}\mid \t(\G)=\set{\a}}=\mset{\Phi_{x}\mid x\in W, \a\in \Phi_{x}}$ 
\item $\mset{\G\in \CL_{s}\mid \t(\G)=\set{-\a}}=\mset{s(\Phi_{x})\mid x\in W, \a\in \Phi_{x}}$ 
\item $\mset{\G\in \CL_{s}\mid \t(\G)=\set{\a,-\a}}=\mset{\Phi_{x}\cup\set{-\a}\mid x\in W, sx<x }$ 
  \end{num}\end{lem*}
  \begin{proof} Note first that $\Psi_{+}:=s(\Phi_{+})=(\Phi_{+}\sm\set {\a})\cup\set{-\a}$ is another positive system for  $\Phi$, with simple roots $s(\Pi)$. Also, $\L=\Phi_{+}\cup\Psi_{+}$.
  It follows readily that if $\G\seq \L$, then $\G$    is closed (resp., biclosed) in $\L$, if and only if  $\G\cap \Phi_{+}$ is closed (resp., biclosed) in $\Phi_{+}$ and $\G\cap \Psi_{+}$ is closed (resp., biclosed) in $\Psi_{+}$.
  So by Lemma \ref{ss3.1}(d),  for any $\G\seq \L$, one has $\G\in \CL_{s}$ if and only if  there are  $x,z\in W$ with $\G\cap \Phi_{+}=\Phi_{x}$ and $\G\cap \Psi_{+}=s_{\a}(\Phi_{z})$. In that case, $\a\in \Phi_{x}$ (resp.,
  $\a\in \Phi_{z}$) if and only if  $\a\in \t(\G)$ (resp., $-\a\in \t(\G)$). 
  
 The inclusions ``$\seq$'' will be proved first. Let $\G\in \CL_{s}$ and let $x,z$ be as above. 
  If $\t(\G)=\set{\a}$ (resp., $\t(\G)=\set{-\a}$) the above immediately shows $\G=\Phi_{x}\ni \a$  (resp., $\G=s(\Phi_{z})$  with $\a\in\Phi_{z}$)  which  gives the inclusion ``$\seq$'' in (b)--(c). For  ``$\seq$'' in (a), assume that $\G\in \CL_{s}$ with $\t(\G)=\eset$. Then $\a\not\in \Phi_{x}$, $\a\not\in \Phi_{z}$ and $\Phi_{x}=s(\Phi_{z})$.
  Hence $\Phi_{sz}=s(\Phi_{z})\dotcup\set{\a}  =\Phi_{x}\dotcup\set{\a}$.
  Therefore $\Phi_{x\vee s}=\ol{\Phi_{x}\cup\set{\a}}=\Phi_{sz}$.  Lemma \ref{ss9.1} gives  $z<sz$ and $\Phi_{sz}=\Phi_{z}\dotcup\set{\a}$, so $x=z$ and $x<sx$, proving ``$\seq$'' in (a). Next, the  proof of the inclusion  ``$\seq$'' in (d) is given. In this case,  $\a\in \Phi_{x}$, $\a\in \Phi_{z}$,
 and  $\Phi_{x}\sm\set{\a}=\G\cap \Phi_{+}\cap \Psi_{+}=s(\Phi_{z})\sm\set{-\a}=\Phi_{sz}$. 
 This gives $sz<x$, $sz\vee s=x$ and $\Phi_{x}\sm\set{\a}=\Phi_{sx}$, so $sx=sz$, $x=z$, $sx<x$ and $\G=\Phi_{x}\cup\set{-\a}$.  
 
Next, the reverse inclusions ``$\sreq$''  are proved, again using the characterization of elements of $\CL_{s}$ in the first paragraph of the proof. For (a), suppose $\G=\Phi_{x}$ where $x<sx$.
 By Lemma \ref{ss9.1}, $\Phi_{x}=s(\Phi_{x})$. So $\G\cap \Phi_{+}=\Phi_{x}$, and $\G\cap \Psi_{+}=s(\Phi_{x})$.
 For (b), suppose $\G=\Phi_{x}$ where $\a\in \Phi_{x}$.
 Then $\G\cap \Phi_{+}=\Phi_{x}$ and $\G\cap \Psi_{+}=\Phi_{x}\sm\set{\a}=s(\Phi_{sx})$.
 For (c), suppose $\G=s(\Phi_{x})$ where $\a\in \Phi_{x}$. Then $\G\cap \Psi_{+}=s(\Phi_{x})$
 and $\G\cap \Phi_{+}=s(\Phi_{x})\sm\set{-\a})=\Phi_{sx}$. Finally, for (d) suppose that
 $\G=\Phi_{x}\cup\set{-\a}$ where $sx<x$. From Lemma \ref{ss9.1}, it follows that  $\Phi_{sx}=s(\Phi_{sx})=\Phi_{x}\sm\set{\a}$.
 Hence $\G\cap \Phi_{+}=\Phi_{x}$ and $\G\cap  \Psi_{+}= (\Phi_{x}\sm\set{\a})\cup \set{-\a}=s(\Phi_{x})$. This completes the proof in all cases.
   \end{proof}
   \subsection{Proof of Theorem \ref{ss1.11}} Part  (a) will be proved first.  Note for any $\G\in \CL$, there are only finitely many $\D\in \CL$ with $\D\seq \G$ (since $\G$ is finite).
   Hence any subset $X$  of $\CL$ with an upper bound is finite.   So by induction, one is reduced to proving the result for joins $\G\vee\D$ of two elements
   (when the join  exists). More precisely, it has to be shown that 
   \be \lb{eq9.4.1} \G\vee \D=\ol{\G\cup\D}\ee
   whenever $\G,\D$ have an upper bound $\Sigma$ in $\CL$.  Note that obviously $\Sigma \sreq \ol{\G\cup\D}$ for any such upper bound (in particular, for $\Sigma=\G\vee \D$ if the join exists). 
   To prove \eqref{eq9.4.1},  there are sixteen cases depending on the types $\t(\G)$ and $\t(\D)$.
   To facilitate reductions of some of the  cases to others, \eqref{eq9.4.1} will first be proved in the special case that $\D=\set{\a}$.  The symmetry $\G\longleftrightarrow \D$ given by interchanging $\G$ and $\D$
   and the symmetry given by the $W(\CL)$-action will also be used to reduce the number of cases to be considered.
   
   Suppose then that $\D=\set{\a}$. If $\a\in \G$, then $\G\vee \D=\G=\ol{\G\cup \D}$, it may be  assumed that $\t(\G)\seq\set{-\a}$. If $\t(\G)=\eset$, then
    $\G=\Phi_{x}$ where $x<sx$. Then $\ol{\G\cup\D}=\ol{\Phi_{x}\cup\set{\a}}= \Phi_{x}\cup\set{\a}=\Phi_{sx}$. Since $\Phi_{sx}\in \CL$, it follows that $\ol{\G\cup\D}=\Phi_{sx}=\G\vee \D$
    in this case.  Next, consider the case that $\t(\G)=\set{-\a}$.
    In that case,  $\G=s(\Phi_{x})$ where $\a\in \Phi_{x}$. By assumption, there is  is some $\Sigma\in \CL$ with $\G\cup \D\seq \Sig$. Necessarily, $\t(\Sigma)=\set{\a,-\a}$ so $\Sigma=\Phi_{z}\cup\set{-\a}$ where
    $sz<z$. Now $\G=s(\Phi_{x})=\Phi_{sx}\cup\set{-\a}$.
    Hence \bee \ol{\G\cup{\D}}=\ol{\Phi_{sx}\cup\set{\a}\cup\set{-\a}}\seq \Phi_{z}\cup\set{-\a}\eee which implies that $\ol{\Phi_{sx}\cup\set{\a}}\seq \Phi_{z}$. Therefore, $sx\vee s$ exists
    and $\Phi_{sx\vee s}=\ol{\Phi_{sx}\cup\set{\a}}$. By Lemma \ref{ss9.1}, it follows that $sx\vee s=y$ where $sy<y$. So
    \bee \ol{\G\cup{\D}}=\ol{\Phi_{sx}\cup\set{\a}\cup\set{-\a}}\sreq \ol{\Phi_{sx}\cup\set{\a}}\cup\set{-\a}=\Phi_{y}\cup\set{-\a}.\eee
    But $\G\cup\D=\Phi_{sx}\cup\set{\a,-\a}\seq \Phi_{y}\cup\set{-\a}\in \CL$.
    Hence the join of $\G$ and $\D$ exists and is given by $\G\vee \D=\ol{\G\cup\D}=\Phi_{y}\cup\set{-\a}$. This completes the proof in the case that $\D=\set{\a}$. Using the above-mentioned symmetry, \eqref{eq9.4.1} also holds in the cases $\D=\set{-\a}$, $\G=\set{\a}$ or $\G=\set{-\a}$.

   Next, observe the following. Suppose that  for $i=1,2$, $\D_{i}\in \CL$ is such that for all
   $\G\in \CL$ such that $\G$, $\D_{i}$ have an upper bound in $\CL$, they have a join
   $\G\vee \D_{i}=\ol{\G\cup \D_{i}}$. Assume also that $\D_{1},\D_{2} $ have an upper bound,   so $\D_{1}\vee \D_{2}=\ol{\D_{1}\cup\D_{2}}$. For any $\G\in \CL$ for which
   $\G$ and $\D_{1}\vee \D_{2}$ have an upper bound $\Sigma$,  Lemma \ref{ss3.2}(c) implies that
   \bee \ol{\G\cup (\D_{1}\vee \D_{2})}=\ol{\G\cup \D_{1}\cup \D_{2}}=\ol{(\G\vee \D_{1})\cup \D_{2}}=(\G\vee \D_{1})\vee \D_{2}=\G\vee (\D_{1}\vee \D_{2}),\eee  
   noting that $\Sigma$ is also an upper bound for $\G,\D_{1}$ and for $\G\vee \D_{1},\D_{2}$.
   
   Using the previous two paragraphs, it follows  that \eqref{eq9.4.1} holds if 
   $\D\seq \set{\a,-\a}$ or $\G\seq\set{\a,-\a}$. The next step is to reduce to the case that $\t(\G)=\t(\D)$.  Suppose in general  that $\G,\D\in \CL$ have an upper bound $\Sigma$. Set
   $\Xi:=\ol{\t(\G)\cup\t(\D)}\seq\set{\a,-\a}$. Note $\Xi\in \CL$. Then $\Sigma$ is an upper bound for $\G$, $\D$ and $\Xi$. Set $\G':=\G\vee \Xi=\ol{\G\cup\Xi}$ and $\D'=\D\vee \Xi=\ol{\D\cup\Xi}$. Then
    $\t(\G')=\t(\D')=\Xi$,
    $\ol{\G\cup \D}=\ol{\G'\cup\D'}$, and $\Sigma$ is an upper bound for $\G',\D'$.
   If it is known that  $\G',\D'$ have a least upper bound  $\G'\vee  \D' =\ol{\G'\cup \D'}$, it follows that $\G,\D$ have the join $\G\vee \D=\ol{\G\cup\D}$. That is, the proof of
   \eqref{eq9.4.1} is reduced  to its special case in which $\t(\G)=\t(\D)$.
   
   So now suppose $\G,\D\in \CL$  have an upper bound $\Sigma$ and are arbitrary except that $\t(\G)=\t(\D)=\Xi$. Write $\Sigma\cap \Phi_{+}=\Phi_{z}$ where $z\in W$.
  It is necessary to consider four cases, according to the value of $\Xi$.
   The first case is that in which $\Xi=\eset$. Write $\G=\Phi_{x}$ and $\D=\Phi_{y}$ where $x<sx$ and $y<sy$. Then $x\in \set{s}^{*}$ and $y\in \set{s}^{*}$.
  Observe that  $\Phi_{x},\Phi_{y}\seq\Phi_{z}$. Hence $w:=x\vee y$ exists in $(W,\leq)$, with $\Phi_{w}=\ol{\Phi_{x}\cup \Phi_{y}}$.
  From Lemma \ref{ss1.9}, it follows that $w\in \set{s}^{*}$, so $w<sw$ and $\Phi_{w}\in \CL$.
  Hence $\ol{\G\cup\D}=\Phi_{w}=\G\vee \D$, completing the proof in this case.
  
  The second case we consider is that in which $\Xi=\set{\a}$. Here, we may write 
  $\G=\Phi_{x}$, $\D=\Phi_{y}$ where $\a\in \Phi_{x}\cap \Phi_{y}$.
  Then $x,y\leq z$ so
  $w:=x\vee y$ exists with $\Phi_{w}=\ol{\Phi_{x}\cup\Phi_{y}}$.
  Since $\a\in \Phi_{w}$, this immediately implies that $\Phi_{w}\in \CL$ and 
  $\ol{\G\cup\D}=\Phi_{w}=\G\vee \D$ as required.
  
  The third case, in which $\Xi=\set{-\a}$ reduces immediately to the second case by using the symmetry given by the action of $s$ on $\CL$. The final case is that in which $\Xi=\set{\a,-\a}$.
  Here, write $\G=\Phi_{x}\cup\set{-\a}$ and $\D=\Phi_{y}\cup\set{-\a}$ where $sx<x$ and $sy<y$. Then $sx\in\set{s}^{*}$ and $sy\in \set{s}^{*}$. Also, $\Phi_{sx},\Phi_{sy}\seq \Phi_{z}$ so $sx\vee sy$ exists. Write $sx\vee sy=sw$, so $\Phi_{sw}=\ol{\Phi_{sx}\cup\Phi_{sy}}$ and $sw\in \set{s}^{*}$ by Lemma \ref{ss1.9}. In particular, $sw<w$ so
  $\Phi_{w}\cup\set{-\a}\in \CL$. Now  \bee \ol{\G\cup\D}=\ol{\Phi_{sx}\cup\Phi_{sy}\cup\set{\a}\cup\set{-\a}}\sreq\Phi_{sw}\cup\set{\a,-\a}=\Phi_{w}\cup\set{-\a}\in \CL.\eee
  On the other hand, $sx,sy\leq sw$ imply $x,y\leq w$ and  $\G\cup\D\seq \Phi_{w}\cup\set{-\a}$.
  It follows that $\ol{\G\cup \D}=\G\vee \D$. This completes the proof of (a).
  
  Using Lemma \ref{ss4.3}, it follows  that $\CL$ is a complete meet semilattice.
  The proof of (b) involves a few additional facts supplementing those of Lemma \ref{ss9.1}. First, one checks using Lemma \ref{ss9.1}(a) that \be x\in s^{*}\iff s(\Phi'_{sx})=\Phi'_{sx}\iff 
   \Phi_{x}'=\Phi'_{sx}\dotcup\set{\a}.\ee  Using this and  Corollary
   \ref{ss6.2}, it follows    that 
   for $x,y,z\in W$, if $\ol{\Phi_{x}\cup\Phi'_{y}}=\Phi'_{z}$, $x\in s^{*}$ and $sy\in s^{*}$ then $sz\in s^{*}$.
  To prove (b), note first that the meet of  any non-empty subset  $X$ of $\CL$ is equal to the (directed) intersection of the family of meets of its finite subsets.
  Since $\CL$ has a minimum element and finite intervals,   the proof of (b) easily  reduces to that of its special case for meets of finite subsets, and then by induction it reduces to the case of meets of pairs of elements of $\CL$. The proof of this is very similar
     to that of (a). It is required  to show that for $\G,\D\in \CL$, we have  \be\lb{eq9.4.2}  \L\sm(\G\wedge \D)=\ol{(\L\sm \G)\cup(\L\sm \D)}.\ee
     Again, there are $16$ possible cases initially. To reduce the number of cases, one first shows that for $\G\in \CL$ and $\Xi\seq\set{\a,-\a}$ one has  
     \be\lb{eq9.4.3} \ol{(\L\sm \G)\cup \Xi}=\L\sm \D\ee
for some $\D\in \L$. This is trivial for $\Xi=\eset$. One  checks \eqref{eq9.4.3} first for $\Xi=\set{\a}$ using Lemma \ref{ss9.2}, as in (a). Next,  \eqref{eq9.4.3} follows using the $W(\L)$-action for $\Xi=\set{-\a}$, and finally it follows  for $\Xi=\set{\a,-\a}$ on 
writing
 \bee \ol{(\L\sm \G)\cup \set{\a,-\a}}=\ol{ (\L\sm \G)\cup \set{\a}\cup\set{-\a}}=\ol{(\L\sm \D_{1})\cup\set{-\a}} =\L\sm \D\eee
 for some $\D_{1},\D$ in $\CL$. Finally, using \eqref{eq9.4.3} one reduces \eqref{eq9.4.2} to the case $\t(\G)=\t(\D)$ (as in (a)) which one checks (also as in (a), but
 using Lemma \ref{ss1.6a} instead of Corollary \ref{ss1.6} where necessary).
     The proof of (b) is in fact  slightly simpler  than that of (a) since the meets in $W$ involved in the proof of (b) automatically exist, whereas the existence of the necessary  joins in $W$  had to be checked. Further details are omitted. 
  
  Finally, it remains to prove (c). Suppose that $X\seq \CL$ has an upper bound.
  Then \bee \t(\join X)=\t(\overline{\cup_{\G\in X}\G})=\ol{\t({\cup_{\G\in X}\G})}=\ol{\cup_{\G\in X}\t(\G)}=\join_{\G\in X}\tau({\G})\eee
   The proof for meets in (c)  is similar and is omitted.
   
   \ssect{Join of a finite biclosed set and a finite biclosed set} \label{ss9.5} The following additional special case of Conjecture \ref{ss1.14} is proved along similar lines as the proof of Theorem \ref{ss1.11} (in fact, a critical special case of it was already required in the proof of  \ref{ss1.11}(b)), and details of  its  proof are omitted.
   \begin{prop*} Let notation be as in Theorem $\text{\rm \ref{ss1.11}}$.
   If  $\G,\Xi\in \CL$ then 
    $\ol{(\L\sm \G)\cup \Xi}=\L\sm \D$
for some $\D\in \L$.\end{prop*}

\section{Variants for other closure operators} \lb{S10}
Given a closure operator $c$ on $\Phi$, one may define $c$-coclosed, $c$-biclosed subsets of $\Phi_{+}$ etc in the same way as for $2$-closure, and ask whether the analogues of the main results (e.g. in Section 1)  hold with all $\ol{\G}$ replaced by $c(\G)$, ``closed'' replaced by ``$c$-closed,''  ``biclosed'' replaced by ``$c$-biclosed'' etc.  This section
 indicates, somewhat informally,  what can be proved  about
two other such closure relations by simple modifications of the arguments in earlier sections.

 We first introduce some more  terminology concerning   closure operators. \ssect{Terminology for closure operators}\lb{ss10.1a} 
A closure operator $c$ on a set $X$ is said to be of finite character if for all $A\seq X$, \[c(A)=\cup_{\substack{A_{0}\seq A\\ \vert A_{0}\vert <\aleph_{0} }}c(A_{0}).\]

Say that a closure operator $c$ on $X$ is an antiexchange closure  operator if  
for $A\seq X$ and all $x,y\in X\sm c(A)$, $x\in c(A\cup\set{y})$ implies $y\not\in c(A\cup\set{x})$ (this terminology is often restricted to the case of finite $X$, as in \cite{Edel} but we shall not do so here).

Consider a  family of root systems of a Coxeter group $(W,S)$ 
for each one $\Psi$ of which there is an associated closure operator on $\Psi$.    Say the family of closure operators is combinatorial  if the closure operator on $T\times \set{\pm 1}$ defined by transport of structure  using the canonical bijection $\Psi\xrightarrow{\cong}T\times \set{\pm 1}$ given by $\epsilon \alpha\mapsto (s_{\a},\e)$ for $\alpha\in \Psi_{+},\e\in \set{\pm 1}$,  is  independent of the choice of  root system $\Psi$ in the family.   Of course, other, quite different  definitions of combinatorial closure operators could be made; the above is convenient for our purposes  here.

\subsection{$\mathbb{Z}$-closure on finite crystallographic root systems} \lb{ss10.1}
Let $\Psi$ be a (reduced) crystallographic root system of a finite Weyl group as in \cite{Bour}.
There is a standard closure operator on $\Psi$, which we call $\bbZ $-closure to avoid confusion with $2$-closure, for which the $\bbZ $-closed sets $\G$ are those
for which $\a,\b\in \Psi$ and $\a+\b\in \Phi$ implies $\a+\b\in\Psi$.
Equivalently $\Psi$ is $\bbZ $-closed if $\a,\b\in \Psi$ and $m\a+n\b\in \Phi$ with $m,n\in \bbN$  implies $m\a+n\b\in\Psi$ (as one sees by reduction to rank two; see also \cite{Pil}).
The  $\bbZ $-closure has well known natural interpretations in the context of semisimple complex Lie algebras, for instance. 

Both $2$-closure and $\bbZ $-closure are closure operators of finite character, trivially. However,  $2$-closure  differs from 
$\bbZ $-closure  in some significant respects (aside from its obvious applicability to more general classes of Coxeter groups).
For example, it is shown in \cite{Pil} that $\bbZ$-closure restricted to the set $\Psi_{+}$ of  positive roots  of $\Psi$ is an anti-exchange closure operator   whereas $2$-closure on $\Phi_{+}$ is not anti-exchange    for $(W,S)$ of type $F_{4}$, $H_{3}$ or $H_{4}$.  Also,
  $2$-closure is easily seen to be combinatorial, but  $\bbZ $-closure is not combinatorial in general (e.g in type $B_{2}$).

\ssect{Analogues for $\mathbb{Z}$-closure of some of the results for $2$-closure}\lb{ss10.2}    This subsection discusses  the extent to which the main results and conjectures of this paper  are known to  apply to finite Weyl groups  with $\bbZ $-closure on $\Psi$ in place of  $2$-closure.
\begin{prop*} The  $\bbZ $-closure  analogue of $\text{\rm Theorem \ref{ss1.5}}$ is true.\end{prop*}
\begin{proof}  According to \cite[Ch VI, \S 1, Ex 16]{Bour}, a subset of $\Psi_{+}$ is $\bbZ $-biclosed if and only if  it is of the form $\Psi_{w}:=\Psi_{+}\cap w(-\Psi_{+})$ for some $w\in W$. 
  The proof of Theorem \ref{ss1.5}(a) applies mutatis mutandis to establish the analogue for $\bbZ $-closure of \ref{ss1.5}(a); then the $\bbZ $-closure analogue \ref{ss1.5}(b) holds by the $\bbZ $-closure analogue of the argument for the proof of Corollary  \ref{ss4.4}(c) (or by using Remark \ref{ss4.4}).\end{proof}

 The analogues for $\bbZ $-closure of Lemma \ref{ss1.7}(a)--(b) fail for  $W$ of type $B_{2}$. It is asserted in \cite{Mal} that the $\bbZ $-closure analogue of Conjecture \ref{ss1.14}(c) in the case $\L=\Psi$ holds
 (though there is a gap in the published proof; see \cite{DCH}. I thank Eugene Karolinsky for these references).
  It is not immediately clear  if the $\bbZ $-closure analogues of  Theorem \ref{ss6.1}, \ref{ss1.8},  or  \ref{ss1.11} hold.

\ssect{The convex geometric closure operator $d$}\lb{ss10.3}  Another  natural closure operator $d$  on $\Phi$ is given by  $d(\Gamma):= \Phi\cap\bbR_{\geq 0}\G$. The  operator $d$ is an anti-exchange  closure operator of finite character, but it  is not combinatorial
in general. In fact,  for    infinite $W$ of rank  four with no braid relations (i.e all entries of its Coxeter matrix are  either $1$ or $\infty$)  there are many possible root systems (in the class \cite{Sd}, with the inner product   normalized so $\mpair{\a,\a}=1$ for all $\a\in \Pi$) with linearly independent simple roots $\Pi$, determined  by arbitrary  choices of inner products $\mpair{\a,\b}=\mpair{\b,\a}\leq -1$ for distinct $\a,\b\in \Pi$. It  is easy to check that  the closure operators on $T\times\set{\pm 1}$ corresponding to the resulting closures $d$ as above, genuinely depend on the choice of root system. \begin{rem*} The operators $d$ for root systems of a  Coxeter system $(W,S)$ of rank three are combinatorially invariant, by an argument  involving  homotopies of root systems
and the combinatorial nature of $d$-closure restricted to the maximal 
dihedral root subsystems. Informally, any two root systems in the class 
are connected by a homotopy, determined by a suitable homotopy
from one matrix $(\mpair{\alpha,\beta})$ to another in the space of such 
matrices attached to root systems. As the root system varies in such a 
homotopy, a root can never enter or leave a plane spanned by two 
other roots (since the reflection in roots of this plane generate a maximal dihedral reflection subgroup, and the sets of reflections of such subgroups are completely determined by the Coxeter system). More precisely, in the rank three case, the oriented matroid 
closure operators on $T\times\set{\pm 1}$ obtained by transfer of 
structure from the various root systems all coincide; this argument  uses 
a characterization of finite rank  (possibly infinite) oriented matroids by 
their basis orientations (cf. \cite[Exercise 3.13]{BVSWZ}).

There is an obvious question as to   whether combinatorial invariance in this sense extends to oriented geometry root systems (with definition as suggested in \ref{matroid}) of rank three Coxeter systems.\end{rem*}

\ssect{Coverings in Bruhat order,  and polyhedral cones}\lb{ss10.4} The following fact from
 \cite[Proposition 3.6]{DyKLS} will be used below.
  \begin{lem*} For any $x\in W$, $\Phi_{x,1}$ (resp., $\Phi_{x,-1}$) is a set of representatives of the extreme rays of a pointed polyhedral cone $\bbR _{\geq 0}\Phi_{x}'$ (resp., $\bbR _{\geq 0} \Phi_{x}$) spanned by $\Phi'_{x}$ (resp., $\Phi_{x}$). Further, $\bbR _{\geq 0} \Phi_{x}\cap \bbR _{\geq 0} \Phi'_{x}=\set{0}$.
\end{lem*}
\ssect{Analogues for $d$ of some results for $2$-closure}\lb{ss10.5} Recall that $d$-coclosed sets and  $d$-biclosed subsets are defined in the obvious way. 
 \begin{prop*} The $d$-closure analogues of
 $\text{\rm Theorem \ref{ss1.5}}$, $\text{\rm Lemma \ref{ss1.7}}$,
 $\text{\rm Theorem \ref{ss1.8}}$ and  $\text{\rm Proposition \ref{ss6.2}}$ are all true.
 \end{prop*}
 \begin{proof}
 Any $d$-closed set is obviously closed. 
 Hence, for any $\G\seq\Phi_{}$, we have $\ol{\G}\seq d(\G)$.  Also, $\Phi_{w}$ and $\Phi'_{w}$ are $d$-biclosed for $w\in W$; this holds since $\Phi_{+}$ and $-\Phi_{+}$ can be strictly separated by a (linear) hyperplane,
and hence so can $w(\Phi_{+})\supseteq \Phi_{w}'$ and $w(-\Phi_{+})\supseteq \Phi_{w}$.
Hence if $\G\subseteq \Phi$ is  such that $\ol{\G}$ is $d$-closed
(e.g. $\ol{\G}=\Phi_{w}$ or $\ol{\G}=\Phi'_{w}$)
then $\ol{\G}=d(\G)$.

 Using the previous paragraph, one sees that the analogues for $d$ of Theorems \ref{ss1.5}  and \ref{ss1.8} and Proposition \ref{ss6.2} hold mutatis mutandis. In each case, the $d$-analogue follows from the corresponding statement for $2$-closure.
For example, suppose $X\seq W$ has an upper bound. Then by Theorem \ref{ss1.5}(a), $\ol{\cup_{x\in X}\Phi_{x}}=\Phi_{y}$ where $y=\join X$,  and so by above, $d({\cup_{x\in X}\Phi_{x}})=\Phi_{y}$ i.e. we have proved the $d$-analogue of \ref{ss1.5}(a).

The $d$-analogue of  Lemma \ref{ss1.7} is proved using Lemma \ref{ss10.4}.
In fact, that Lemma implies that for $\a\in \Phi_{x,1}$ (resp., $\a\in \Phi_{x,-1}$), one has 
$\a\not \in d(\Phi'_{x}\sm\set{\a})$ (resp., $\a\not \in d(\Phi_{x}\sm\set{\a})$). 
On the other hand, $d(\Phi_{x,1})=\Phi'_{x}$  (resp., $ c(\Phi_{x,-1})=\Phi_{x}$) holds by the  Lemma and the first paragraph above.
These remarks  easily imply that the $d$-analogue of  Lemma \ref{ss1.7} holds. 

\end{proof}
\begin{rem*} 

(1) One can show  that $\CL_{s}$ in Theorem \ref{ss1.11} is the set  of 
all finite $d$-biclosed subsets of $\L$ (one uses  simple arguments in convex geometry beginning with  Lemma \ref{ss9.3} and the possibility of separating  $\Phi_{w}$ and $\Phi'_{w}$ by a linear hyperplane).
Then   the $d$-closure analogues of 
Theorem \ref{ss1.11} and    Proposition \ref{ss9.5} follow from the  corresponding results for $2$-closure exactly as above.

(2) The analogue for $d$ of Theorem \ref{ss6.1} cannot be deduced
from Theorem \ref{ss6.1} in the same way as the other results above   (since if $\G\seq \Phi_{+}$ is $d$-coclosed and $d(\G)$ has finite complement in $\Phi_{+}$, one can not  deduce that $\ol \G$ has finite complement and   apply Theorem \ref{ss6.1}; one only knows $\ol{\G}\seq d(\G)$).  However, a simple argument  involving  convex geometry shows that if $\G$ is $d$-coclosed and $\a\in \Pi\sm \G$ then  $s_{a}\cdot \G$ is $d$-coclosed; using this and the $d$-analogue of Lemma \ref{ss1.7}, one sees that Theorem \ref{ss6.1} and its proof also hold mutatis mutandis  for  $d$.
\end{rem*}

\def\Dbar{\leavevmode\lower.6ex\hbox to 0pt{\hskip-.23ex \accent"16\hss}D}


\begin{thebibliography}{10}

\bibitem{BjBr}
Anders Bj{\"o}rner and Francesco Brenti.
\newblock {\em Combinatorics of {C}oxeter groups}, volume 231 of {\em Graduate
  Texts in Mathematics}.
\newblock Springer, New York, 2005.

\bibitem{BjEZ}
Anders Bj{\"o}rner, Paul~H. Edelman, and G{\"u}nter~M. Ziegler.
\newblock Hyperplane arrangements with a lattice of regions.
\newblock {\em Discrete Comput. Geom.}, 5(3):263--288, 1990.

\bibitem{BVSWZ}
Anders Bj{\"o}rner, Michel Las~Vergnas, Bernd Sturmfels, Neil White, and
  G{\"u}nter~M. Ziegler.
\newblock {\em Oriented matroids}, volume~46 of {\em Encyclopedia of
  Mathematics and its Applications}.
\newblock Cambridge University Press, Cambridge, 1999.

\bibitem{Bour}
N.~Bourbaki.
\newblock {\em \'{E}l\'ements de math\'ematique. {F}asc. {XXXIV}. {G}roupes et
  alg\`ebres de {L}ie. {C}hapitre {IV}: {G}roupes de {C}oxeter et syst\`emes de
  {T}its. {C}hapitre {V}: {G}roupes engendr\'es par des r\'eflexions.
  {C}hapitre {VI}: syst\`emes de racines}.
\newblock Actualit\'es Scientifiques et Industrielles, No. 1337. Hermann,
  Paris, 1968.

\bibitem{Br}
Brigitte Brink.
\newblock On centralizers of reflections in {C}oxeter groups.
\newblock {\em Bull. London Math. Soc.}, 28(5):465--470, 1996.

\bibitem{HowBr}
Brigitte Brink and Robert~B. Howlett.
\newblock Normalizers of parabolic subgroups in {C}oxeter groups.
\newblock {\em Invent. Math.}, 136(2):323--351, 1999.

\bibitem{BF}
J.~Richard B{\"u}chi and William~E. Fenton.
\newblock Large convex sets in oriented matroids.
\newblock {\em J. Combin. Theory Ser. B}, 45(3):293--304, 1988.

\bibitem{CH}
M.~Cuntz and I.~Heckenberger.
\newblock Weyl groupoids with at most three objects.
\newblock {\em J. Pure Appl. Algebra}, 213(6):1112--1128, 2009.

\bibitem{DavPr}
B.~A. Davey and H.~A. Priestley.
\newblock {\em Introduction to lattices and order}.
\newblock Cambridge University Press, New York, second edition, 2002.

\bibitem{DyHS2}
M.~J. Dyer.
\newblock Hecke algebras and shellings of {B}ruhat intervals. {II}. {T}wisted
  {B}ruhat orders.
\newblock In {\em Kazhdan-Lusztig theory and related topics (Chicago, IL,
  1989)}, volume 139 of {\em Contemp. Math.}, pages 141--165. Amer. Math. Soc.,
  Providence, RI, 1992.

\bibitem{DyHS1}
M.~J. Dyer.
\newblock Hecke algebras and shellings of {B}ruhat intervals.
\newblock {\em Compositio Math.}, 89(1):91--115, 1993.

\bibitem{DyKLS}
M.~J. Dyer.
\newblock Bruhat intervals, polyhedral cones and {K}azhdan-{L}usztig-{S}tanley
  polynomials.
\newblock {\em Math. Z.}, 215(2):223--236, 1994.

\bibitem{DyQuo}
M.~J. Dyer.
\newblock Quotients of twisted {B}ruhat orders.
\newblock {\em J. Algebra}, 163(3):861--879, 1994.

\bibitem{DyRig}
M.~J. Dyer.
\newblock On rigidity of abstract root systems of {C}oxeter groups.
\newblock {\em {\tt arXiv:1011.2270 [math.GR]}}, 2010.

\bibitem{DyRef}
Matthew Dyer.
\newblock Reflection subgroups of {C}oxeter systems.
\newblock {\em J. Algebra}, 135(1):57--73, 1990.

\bibitem{DyBru}
Matthew Dyer.
\newblock On the ``{B}ruhat graph'' of a {C}oxeter system.
\newblock {\em Compositio Math.}, 78(2):185--191, 1991.

\bibitem{Sd}
Matthew Dyer and Cedric Bonnaf\'e.
\newblock Semidirect product decompositions of {C}oxeter groups.
\newblock {\em Comm. in Alg.}, 38(4):1549--1574, 2010.

\bibitem{Edel}
Paul~H. Edelman.
\newblock Meet-distributive lattices and the anti-exchange closure.
\newblock {\em Algebra Universalis}, 10(3):290--299, 1980.

\bibitem{Edg}
Tom Edgar.
\newblock Sets of reflections defining twisted {B}ruhat orders.
\newblock {\em J. Algebraic Combin.}, 26(3):357--362, 2007.

\bibitem{HV}
I~Heckenberger and W~Welker.
\newblock Geometric combinatorics of {W}eyl groupoids.
\newblock {\em {\tt arXiv:1003.3231 [math.QA]}}, 2010.

\bibitem{HY}
Istv{\'a}n Heckenberger and Hiroyuki Yamane.
\newblock A generalization of {C}oxeter groups, root systems, and {M}atsumoto's
  theorem.
\newblock {\em Math. Z.}, 259(2):255--276, 2008.

\bibitem{Hum}
James~E. Humphreys.
\newblock {\em Reflection groups and {C}oxeter groups}, volume~29 of {\em
  Cambridge Studies in Advanced Mathematics}.
\newblock Cambridge University Press, Cambridge, 1990.

\bibitem{MacL}
Saunders MacLane.
\newblock {\em Categories for the working mathematician}, volume~5 of {\em
  Graduate Texts in Mathematics}.
\newblock Springer-Verlag, New York, 1998.

\bibitem{Mal}
F.~M. Maly{\v{s}}ev.
\newblock Decomposition of root systems.
\newblock {\em Mat. Zametki}, 27(6):869--876, 988, 1980.

\bibitem{DCH}
D.~{\v{Z}}. {\Dbar}okovi{\'c}, P.~Check, and J.-Y. H{\'e}e.
\newblock On closed subsets of root systems.
\newblock {\em Canad. Math. Bull.}, 37(3):338--345, 1994.

\bibitem{Pil}
Annette Pilkington.
\newblock Convex geometries on root systems.
\newblock {\em Comm. Algebra}, 34(9):3183--3202, 2006.

\end{thebibliography}
\end{document}